\documentclass[titlepage]{amsart}
\usepackage[foot]{amsaddr}
\usepackage{amsmath}
\usepackage{amssymb}
\usepackage{bbm}
\usepackage{mathdots}
\usepackage{dsfont}
\usepackage{todonotes,lscape}
\usepackage{graphicx}
\usepackage{hyperref}
\usepackage{tikz}
\usetikzlibrary{calc}
\usepackage{tabularx}
\usepackage{subcaption}
\usepackage{comment}

\usepackage[backend=biber,style=alphabetic]{biblatex}
\def\RR{{\mathbb R}}
\def\SS{{\mathbb S}}

\def\ZZ{{\mathbb Z}}

\def\cal{\mathcal}

\def\a{\alpha}

\def\d{\delta}

\def\l{\lambda}
\def\L{\Lambda}

\let\temp\phi
\let\phi\varphi
\let\varphi\temp

\def\vn{\varnothing}

\def\ones{\mathbbm{1}}

\def\conv{\operatorname {conv}}

\def\supp{\operatorname {supp}}

\def\spanset{\operatorname {span}}

\def\cut{\operatorname {CUT}_d^\square}

\theoremstyle{definition}
\newtheorem{definition}{Definition}[section]
\newtheorem{example}[definition]{Example}

\theoremstyle{plain}
\newtheorem{theorem}[definition]{Theorem}
\newtheorem{lemma}[definition]{Lemma}
\newtheorem{proposition}[definition]{Proposition}
\newtheorem{corollary}[definition]{Corollary}
\newtheorem*{corollary*}{Corollary}

\newtheorem{conjecture}[definition]{Conjecture}
\newtheorem{remark}[definition]{Remark}

\newcommand{\catherine}[1]{{{\color{black}{{#1}}}}}

\newcommand{\J}{\mathcal{J}}
\usepackage[draft]{fixme}
\fxsetup{theme=color, mode=multiuser, noinline}
\FXRegisterAuthor{cb}{CB}{\color{red}Catherine}
\FXRegisterAuthor{rt}{RT}{\color{blue}Rekha}

\begin{document}

\title{Graphical Designs and Gale Duality}
\author{Catherine Babecki and Rekha R. Thomas}
\address{Department of Mathematics, 
University of Washington,\\ Box 354350, 
Seattle, WA 98195}
\email{cbabecki@uw.edu, rrthomas@uw.edu}
\thanks{
Research partially supported by the Walker Family Endowed Professorship in Mathematics at the University of Washington\\
{\em Corresponding Author}: Rekha R. Thomas\\ 
{\em Affiliation}: Department of Mathematics, University of Washington, Box 354350, Seattle, WA 98195, USA\\ 
{\em Email address}: rrthomas@uw.edu}

%\maketitle
%\vspace{-.2 in}
\begin{abstract}
A graphical design is a subset of graph vertices such that the weighted averages of certain graph eigenvectors over the design agree with their global averages. We use Gale duality to show that positively weighted graphical designs in regular graphs are in bijection with the faces of a generalized eigenpolytope of the graph. This connection can be used to organize, compute and optimize designs. We illustrate the power of this tool on three families of Cayley graphs -- cocktail party graphs, cycles, and graphs of hypercubes -- by computing or bounding the smallest designs that average all but the last eigenspace in frequency order. 
\end{abstract}
\maketitle

% notation: 

% n= #vertices = #eigenvectors
% m = #eigenspaces
% k = k-graphical design
% G = (V,E) :=  graph (vertex set, edge set)
% AD^-1 Laplacian
% A = adjacency matrix
% D = diagonal degree matrix
% I = identity matrix 
% W subset V  = subset of vertices
% U = orthonormal eigenvector matrix.
% \lambda = eigenvalue
% \Lambda = eigenspace
%\vspace{-.05 in}
\section{Introduction}
Graphical designs extend classical quadrature rules to the domain of graphs. Informally, a quadrature rule is a set of points on a domain which represent that domain well in terms of numerical integration.  That is, the integral of a suitably smooth function over the domain equals a weighted sum of the function values at the quadrature points. A graphical design is a subset of vertices of a graph which approximates the graph in a similar sense; the average of suitable functions over the whole graph agrees with the weighted sum of the function's values on the design. 

\begin{figure}[h]
\centering
\begin{tabular}{cc}
\subfloat[]{\includegraphics[scale = .6]{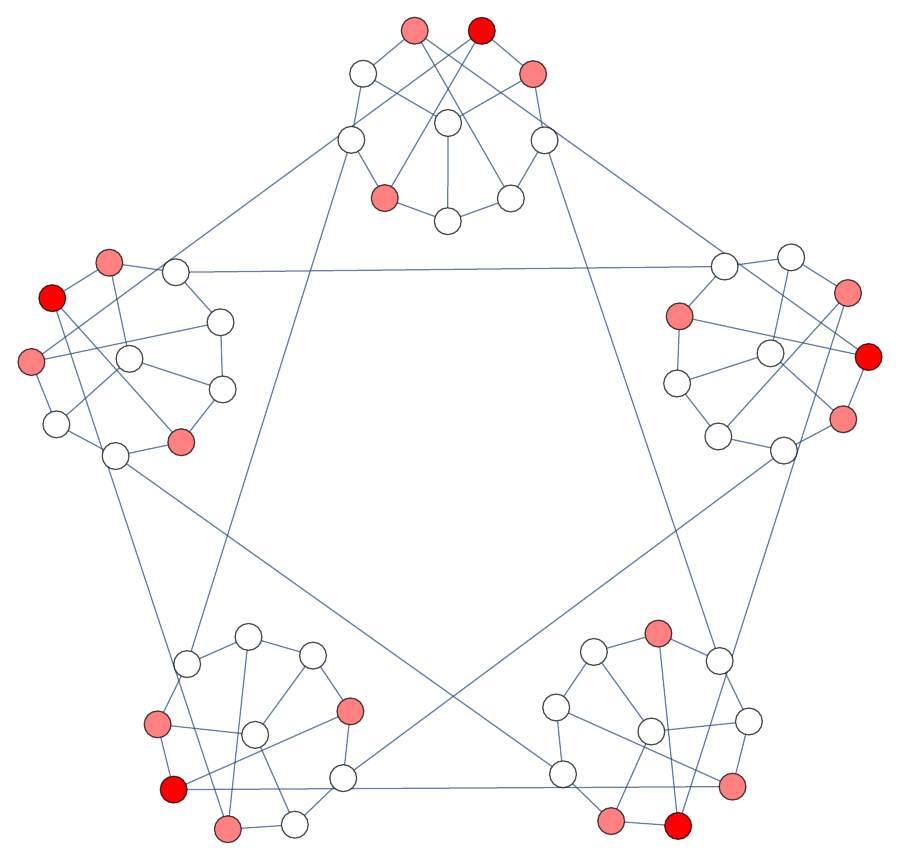}
      } 
        & 
\subfloat[]{\includegraphics[scale = .6]{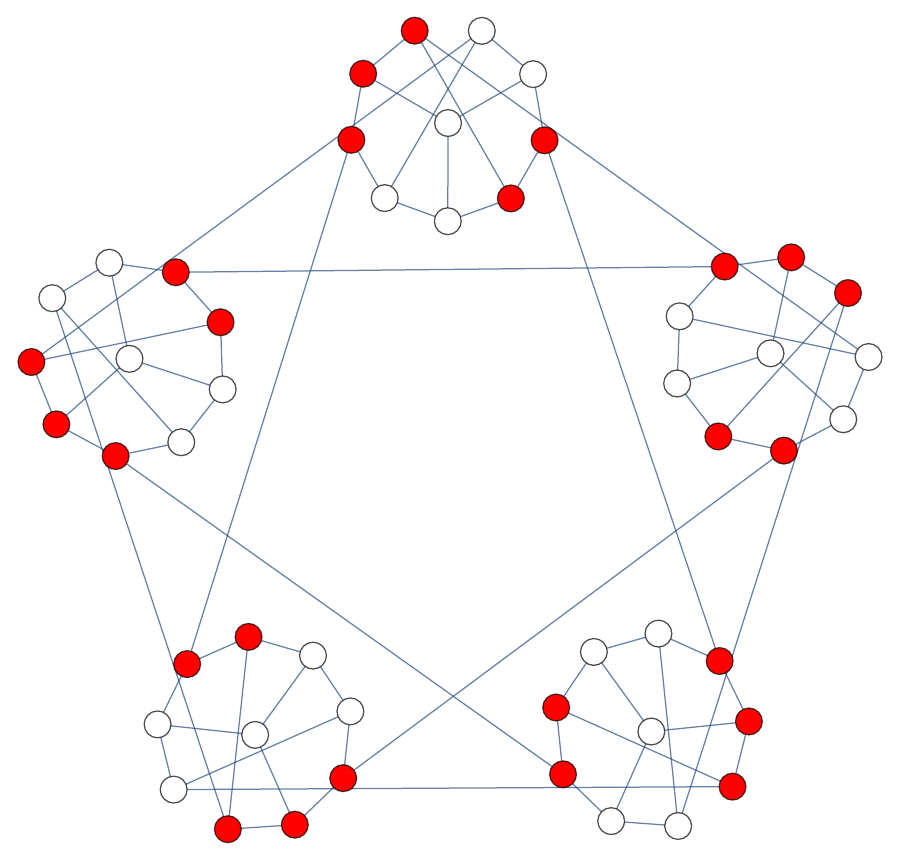}} 
 \end{tabular}
    \caption{The Szekeres Snark has $50$ vertices, $75$ edges and $11$ eigenspaces.
  (A) is a positively weighted $8$-design and (B) is a combinatorial $8$-design. Lighter reds correspond to smaller weights.}
    \label{fig: Szekeres pos and comb}
\end{figure}

In this paper we consider graphical designs in connected regular graphs. 
A function on a graph $G=(V,E)$ is a map $\phi \,:\, V \rightarrow \RR$, which we 
identify with the vector $(\phi(v) \,:\, v \in V) \in \RR^V$. 
The eigenvectors of  the normalized adjacency matrix of a graph $G=(V,E)$ form a basis for all function on $G$. In this paper, we often use the  {\em frequency order} on eigenspaces which is aligned with a notion of ``smoothness'' of functions on $G$.  Given any ordering of the eigenspaces, we define a subset of vertices $W \subseteq V$, with weights $(a_w \,:\, w \in W)$, to be a {\em weighted} $k$-{\em design} in $G$ if for all 
vectors $\phi$ in the first $k$ eigenspaces, $$ \sum_{w \in W} a_w \phi(w) = \frac{1}{|V|} \sum_{v \in V} \phi(v).$$ 
By imposing different requirements on the weights $a_w$, we obtain different types of designs --- weighted ($a_w \in \RR$), positively weighted ($a_w \geq 0$) or combinatorial ($a_w \in \{0,1\}$).  A design is {\em extremal} if it averages all eigenspaces except the last one in the given eigenspace ordering. 
%Precise definitions will be given in Section~\ref{sec:graphical designs}.  
Figure~\ref{fig: Szekeres pos and comb} depicts positively weighted and combinatorial designs in the $3$-regular Szekeres Snark graph that average the first $8$ eigenspaces of the normalized adjacency matrix of this graph in frequency order.

In this paper we show that {\em Gale duality} \cite{GaleOriginal}, from the theory of polytopes, creates a bijection between the positively weighted $k$-designs in a graph 
and the faces of a generalized {\em eigenpolytope} of the graph. 
Eigenpolytopes were defined by Godsil \cite{Godsil_GGP}, and we extend their definition for our purposes.  This connection, and a more general connection to {\em oriented matroid duality}, 
allows one to organize, compute and optimize graphical designs using the combinatorics of the corresponding eigenpolytope. In Figure~\ref{fig:Gale illustration}, we see an illustration of the design-face correspondence for extremal designs on the octohedral graph with the eigenspace for $\l = -1/2$  ordered last. 

\begin{figure}[h]
    \centering
    \includegraphics[scale = .45]{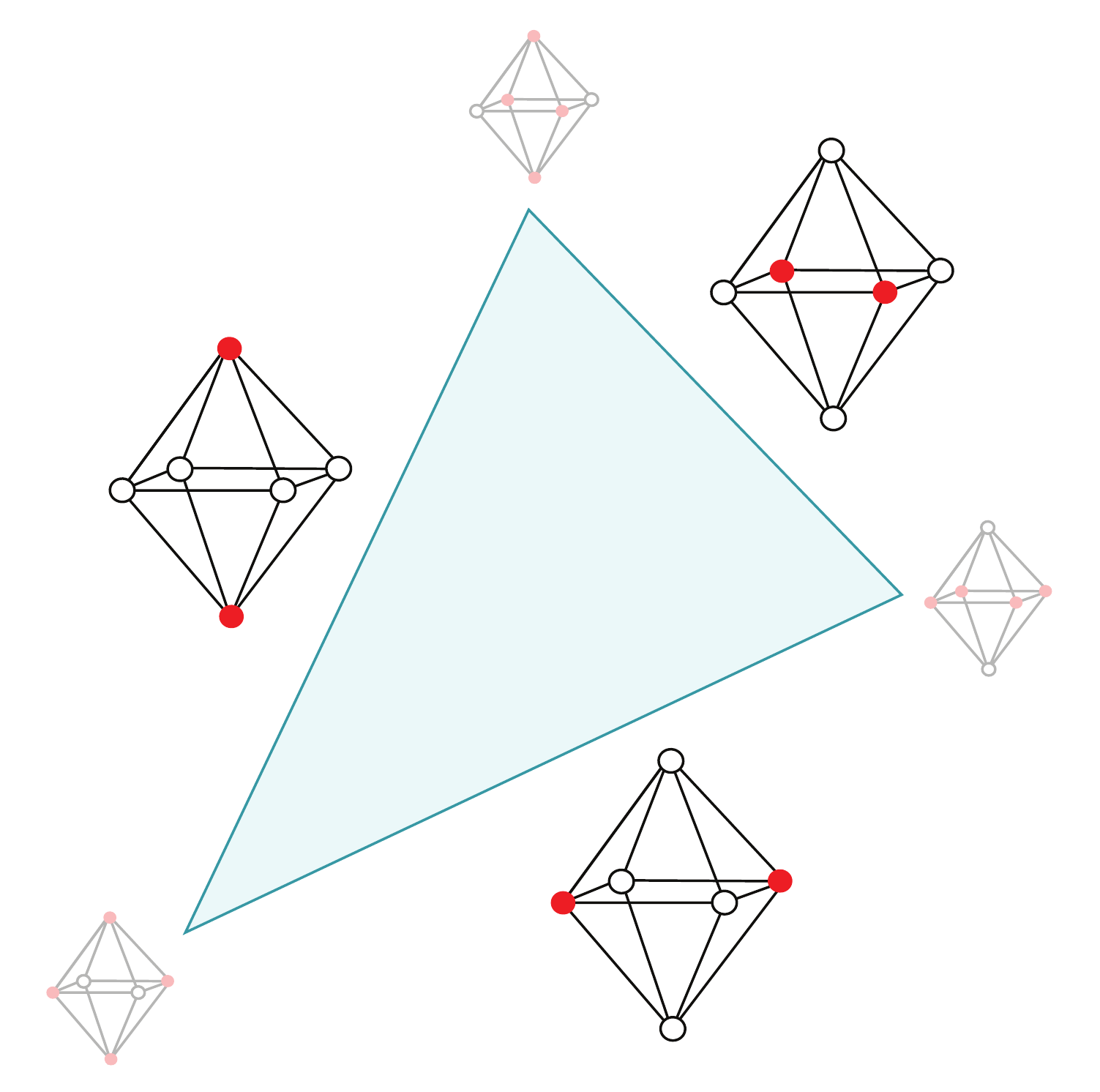}
    \caption{The eigenpolytope for $\l = -1/2$ of the octahedron is a triangle.
    The three facets are in bijection with the three minimal positively weighted designs in the 
    octahedron in the eigenspace ordering that puts $\l = -1/2$ last. The vertices of the triangle 
    correspond to non-minimal designs given by the union of the two adjacent minimal designs.}
    %with each vertex indexed by two graph vertices. 
    %The complements of the facet indices of the triangle give three minimal positively weighted designs. The lower dimensional faces correspond to designs supported on the union of the two adjacent minimal designs.
    \label{fig:Gale illustration}
\end{figure}

We use our main result to compute and/or bound the minimal positively weighted extremal designs in three well-known families of Cayley graphs. For cocktail party graphs, which are the edge graphs of cross-polytopes, we show that every minimal weighted extremal design (in any ordering) is combinatorial, and we describe them explicitly. For cycles, we show that every minimal weighted extremal design in the frequency order is positively weighted, and find their sizes. For the edge graph of a $d$-hypercube and frequency order, we give a precise description of extremal designs for when $d \equiv 2$ mod $4$. For the other congruence classes, we bound the size of a minimal extremal design. The cube results rely on the theory of linear codes. 

Graphical designs were defined relatively recently,  by Steinerberger in \cite{graphdesigns}. 
He was motivated by spherical designs and more generally, designs in manifolds. 
The main result of \cite{graphdesigns} is that if $W$ is a ``good'' graphical design, then either $|W|$ is large, or the $j$-neighborhoods of $W$ grow exponentially. In \cite{SSLinderman}, Steinerberger and Linderman give bounds on the numerical integration error for any quadrature rule on a graph. Golubev in \cite{Golubev} introduced extremal designs and connected them to extremal combinatorics. Babecki in \cite{cubescodes} 
refined the definition of graphical designs to make sense when the eigenvalues of a graph Laplacian have  multiplicity, connected linear codes in the Boolean cube to graphical designs in the edge graphs of hypercubes, and distinguished graphical designs from a handful of related concepts in the adjacent literature. She also hosts a database of examples and code at \url{https://sites.math.washington.edu/~GraphicalDesigns/}.

 Modern data is often best modeled through graphs, driving an increasingly important need for new data processing tools in graphs.  The relatively new field of \textit{graph signal processing} (see, for instance, \cite{ortegaBook, signalprocessingoverview, Pesenson,samplingAGO,samplingTBD,samplingTEOC,samplingBWCNG,SamplingCVSK,SGTgraphwavelets,samplingMSLR,DSPonGraphs}) seeks to extend classical signal processing techniques to the domain of graphs. Graphical designs offer a framework for \textit{graph sampling}, a notoriously difficult problem in applied mathematics. \catherine{A concrete connection between graphical designs and graph sampling was established recently 
 in \cite[Section 3]{RekhaStefanRandomWalks}.}
 Graphical designs also connect to pure mathematics and theoretical computer science through combinatorics, spectral graph theory, error correcting codes, probability, Fourier analysis, and representation theory. 
 
This paper is organized as follows.  Section~\ref{sec:graphical designs} states the formal definitions of designs and illustrates their nuances through examples.  Section~\ref{sec:OM and Eigenpolytopes} introduces the necessary background on Gale duality and proves our main structure theorem connecting graphical designs to the facial structure of eigenpolytopes. We then illustrate the subtleties and power of this result through several further examples.  Section~\ref{sec:OM and Eigenpolytopes} concludes with an overview of the eigenpolytope literature and a rephrasing of the main results of Golubev \cite{Golubev} in terms of eigenpolytopes.
 In Section~\ref{sec:cross-polytopes and cycles}, we 
use Gale duality to classify the minimal positively weighted extremal designs in two families of graphs, the $n$-cycle and cocktail party graphs. Section~\ref{sec:cubes} considers edge graphs of 
$d$-dimensional hypercubes, which we denote by $Q_d$. We describe the eigenpolytopes of $Q_d$, one of which is the cut polytope. 
\catherine{Facets of the cut polytope given by triangle inequalities 
can be used to find the minimum extremal designs in a particular eigenspace ordering of $Q_d$.} 
Under frequency order, we prove upper bounds on the size of a smallest positively weighted extremal design for $Q_d$ using Gale duality and linear codes; these bounds are tight when $d \equiv 2 \mod 4$. 

{\bf Acknowledgments.} We thank Sameer Agarwal and Stefan Steinerberger for many useful discussions and suggestions. We also thank Chris Lee and David Shiroma, undergraduates at the University of Washington, who worked with us in Autumn 2021 on graphical designs. They independently discovered the construction in Bonisoli's theorem on linear codes which provides strong bounds on extremal designs in the graphs of hypercubes. We explain this result in Section~\ref{sec:cubes}.

\section{Graphical Designs: Definitions and Examples}
\label{sec:graphical designs}

Let $G= ([n],E)$ be a connected, simple, undirected graph with vertex set $[n]:= \{1,\ldots,n\}$ and edge set $E$. We will assume throughout that $G$ is 
regular with the degree of every vertex equal to $\d$. The {\em adjacency matrix} $A \in \RR^{n \times n}$ of $G$ is defined by $A_{ij} = 1$ if $ij\in E$ and $A_{ij}=0$ otherwise. Let $D \in \RR^{n \times n}$ be the diagonal matrix with $D_{ii} = \deg(i) =\d$, where
$\deg(i)$ is the degree of vertex $i \in [n]$. Then the spectrum of the {\em normalized adjacency matrix} $ AD^{-1} =(1/\d) A$ is contained in the interval $[-1,1]$, and $1$ is an eigenvalue of $ A D^{-1}$. \catherine{We denote the eigenspace of $1$ by $\L_1$. In general, the dimension of $\L_1$ is the number of connected components of $G$, and in our set up, $\Lambda_1 := \spanset\{\ones\}$, where $\ones$ denotes the all-ones vector.}

Throughout this paper, we will refer to the spectral information of $ A D^{-1}$ as the spectral information of $G$. 
We will use the eigenvalues and eigenspaces of $ A D^{-1}$ to define graphical designs in $G$. 
We note that in \cite{graphdesigns}, graphical designs were defined using the normalized Laplacian matrix $AD^{-1} - I$. Since $AD^{-1}-I$ and $AD^{-1}$ have the same eigenspaces with eigenvalues shifted by $1$, we use the simpler $AD^{-1}$ in this paper. An eigenvector of $AD^{-1}$ will be 
interpreted as a function $ \phi \,:\, V \rightarrow \RR$, with the 
$v$-th coordinate denoted by $\phi(v)$.
We begin by defining what it means for a subset of vertices to {\em average} an eigenspace $\Lambda$ of $G$.

\begin{definition}
Let  $G=([n],E)$ be a graph.
A subset of vertices $W \subseteq [n]$ \emph{averages the eigenspace} $\L$ of $G$ if there are weights $(a_w \in \RR: w \in W)$ such that for every eigenvector $\phi$ in a basis of $\Lambda$,
\begin{equation} \label{eq:average}
\sum_{w \in W}  a_w \phi(w) = \frac{1}{n}\sum_{v \in [n]} \phi(v).
\end{equation} \end{definition}

There are three types of weights of interest in this paper: arbitrary ($a_w \in \RR$), positive ($a_w > 0$), and uniform ($a_w = 1/|W|$). \catherine{  
In classical numerical integration, negative weights are undesirable as they can lead to divergent solutions and numerical instability, see, for instance \cite{HuybrechsInstability}. The main results of this paper are about positively weighted graphical designs.}

If $G$ is regular, then $A D^{-1}$ is symmetric, so we can find a set of $n$ orthogonal eigenvectors that form a basis for $\RR^n$. It will be convenient to not assume that the eigenvectors have unit length, allowing us to use $\ones$ as the eigenvector spanning $\Lambda_1$. The average of $\ones$ over $G$ is \begin{equation} \label{eq: everything averages first eigenspace}
     \frac{1}{n}\sum_{v \in [n]} 1 = 1.
\end{equation}
If $\phi$ is an eigenvector of $ A D^{-1}$ with eigenvalue not equal to $1$, then 
$\ones^\top \phi = 0$ by orthogonality. Hence the average of $\phi$ over $G$ is
\begin{equation} \label{eq:0 average}
 \frac{1}{n}\sum_{v \in [n]} \phi(v) = \frac{1}{n} \ones^\top \phi = 0. 
\end{equation}

Therefore, we may interpret the weights in \eqref{eq:average} 
as a vector $a$ orthogonal to $\phi$ with $a_i = 0$ for all $i \not \in W$ and $a_i = a_w$ for all $i=w \in W$.
Suppose $G$ has $m$ eigenspaces, and fix an ordering with $\L_1 = \spanset \{\ones\}$ ordered first.  
A weighted $k$-graphical design is a subset of vertices that averages the first $k$ eigenspaces in this ordering.

\begin{definition} [$k$-graphical designs]
\label{def:graphical designs}
Suppose $G = ([n],E)$ has $m$  eigenspaces ordered as 
$ \spanset \{\ones\} = \L_1 <\ldots < \L_m.$
\begin{enumerate}
    \item A \emph{weighted $k$-graphical design} of $G$ is a subset $W \subseteq [n]$ and weights $(a_w \in \RR: w\in W)$ such that $W$ averages the eigenspaces $\L_1, \ldots, \L_k$. 
    \item If in addition, $ a_w > 0 $ for all $w\in W$, we call $W$ a \emph{positively weighted $k$-graphical design} of $G$, and 
    \item if $ a_w = 1/|W| $, then we call $W$ a \emph{combinatorial $k$-graphical design} of $G$.
\end{enumerate}

\end{definition}

We often drop the word `graphical' and refer to $k$-designs.
The different types of weights provide a hierarchy of $k$-designs: 
any combinatorial $k$-design is a positively weighted $k$-design, and any positively weighted $k$-design is a weighted $k$-design. In general, the three types of weights provide distinct designs as we will see shortly. 
For later use, it will be convenient to characterize the different types of designs as follows. The {\em support} of a vector $a \in \RR^n$ is $\supp(a) := \{i \in [n] \,:\, a_i \neq 0 \}$. For a subset $W \subset [n]$, define $\ones_W$ by  $\ones_W(i) = 1$ if $i \in W$ and 0 otherwise.

\begin{lemma} \label{lem:computational check}
Suppose $G = ([n],E)$ has $m$ distinct eigenspaces ordered as 
$$ \spanset \{\ones\} = \L_1 <... < \L_m.$$
\begin{enumerate}
    \item $W \subseteq [n]$ is a weighted $k$-design of $G$ if and only if 
    there is a non-zero vector $a \in \RR^n$ such that 
    $$W = \supp(a),\,\,\,\, \phi^\top a = 0 \,\,\,\, \forall \phi \in 
    \Lambda_2, \ldots, \Lambda_k, \,\,\,\, \ones^\top a \neq 0.$$ 
    \item $W \subseteq [n]$ is a positively weighted $k$-design of $G$ if and only if 
    there is a non-zero vector $a \in \RR^n, a \geq 0$ such that 
    $$W = \supp(a) \textup{ and } \phi^\top a = 0 \,\,\,\, \forall \phi \in 
    \Lambda_2, \ldots, \Lambda_k.$$ 
    \item $W \subseteq [n]$ is a combinatorial $k$-design of $G$ if and only if 
      $\phi^\top \ones_W = 0 \,\,\,\, \forall \phi \in 
    \Lambda_2, \ldots, \Lambda_k.$
\end{enumerate}
\end{lemma}

\begin{proof}
The proof of this lemma mostly follows from \eqref{eq:0 average}.  The only extra piece is the condition that $\ones^\top a \neq 0$ in (1). This is because $W = \supp(a)$ averages $\Lambda_1 = \spanset\{\ones\}$ if and only if $ \ones^\top a = 1$. If $a$ is a non-zero vector orthogonal to all vectors in $\L_2, \ldots, \L_k$ for which $\ones^\top a \neq 0$, then we can scale it to get $\ones^\top a = 1$ while preserving the orthogonality requirements. This proves (1). The statements in 
(2) and (3) do not need this condition to be stated explicitly since if 
$a \neq 0$ and $a \geq 0$ or $a \in \{0,1\}^n$ then it follows that $\ones^\top a \neq 0$.
\end{proof}

We note a quick fact about combinatorial designs.

\begin{lemma} \label{lem: comb complements}
If $W \subset [n]$ is a combinatorial $k$-design, then so is $[n] \setminus W.$
\end{lemma}

\begin{proof}
Let $\phi \in \bigcup_{i=2}^k\L_{i}$. 
If $W$ is a combinatorial $k$-design, $\ones_W \in \{0,1\}^n$ and $\phi^\top \ones_W = 0$. Since 
$\phi \perp \ones,$ 
$   \phi^\top \ones_{[n] \setminus W} =  \phi^\top ( \ones - \ones_{W} ) = 0.$
Hence $[n] \setminus W$ is also a combinatorial $k$-design.
\end{proof}

A natural quest at this point is to find the smallest graphical designs 
that can average as many eigenspaces as possible, given a fixed eigenspace ordering. We first note that no proper subset of $[n]$ can average all eigenspaces of $G$.

\begin{lemma}
In a connected, regular graph $G=(V,E)$, no proper subset $W \subset V$ can average all eigenspaces of $G$ 
in any eigenspace ordering with any type of weights.
\end{lemma}

\begin{proof}
Suppose $W \subset [n]$ averages every eigenspace of $G$. Let $U$ be a matrix whose rows form a 
basis for $\Lambda_2 \oplus \cdots \oplus \Lambda_m$. By \eqref{eq:0 average},  
$W=\supp(a)$ for some $a \in \ker U$, which is 1-dimensional and spanned by $\ones$. Therefore, $a = \ones/n$ and $\supp(W) = V$. 
\end{proof}

This brings us to the next two definitions.

\begin{definition}[Maximal and Extremal Designs]
Suppose $G$ has $m$ eigenspaces with a fixed ordering 
$ \spanset \{\ones\} = \L_1 <... < \L_m$, and let $k_{\max}$ be maximal such that $G$ has a $k_{\max}$-graphical design.  
\begin{enumerate}
    \item A \emph{maximal design} in $G$ is a  $k_{\max}$-graphical design. 
    \item  An  \emph{extremal design} in $G$ is an $(m-1)$-graphical design.
\end{enumerate}
\end{definition}

\catherine{Note that the maximal and extremal designs of a graph depend on the eigenspace ordering chosen.}
We show in Section~\ref{sec:OM and Eigenpolytopes} that every graph has a positively weighted extremal design. However, a graph may have no extremal combinatorial designs. 

\begin{example} \label{ex: icosa no extremal}
Let $G = ([12],E)$ be the edge graph of a regular icosahedron.  We record a basis for each eigenspace $\Lambda_i$ of $G$ (with eigenvalue $\lambda_i$) in 
Figure~\ref{fig: icosahedron eigenbasis}.

\begin{figure}[h!]
\begin{tabular}{c|cccccccccccc}
$\l_1 = 1$ & 1 & 1 & 1 & 1 & 1 & 1 & 1 & 1 & 1 & 1 & 1 & 1\\
\hline
& $\varphi$ &  $- \varphi$ &   $- \varphi$ &  $\varphi$  & $-1 $ &   $-1 $  &    1 &    1 &       0    &      0    &     0     &    0  \\
 $ \l_2 = -.4472$ & $-1 $  &  1 &   $\varphi$  &  $- \varphi$ &    0  & $\varphi$  &  $- \varphi$ &       0   &      0  & $-1 $  &    1 &  0 \\
& $\varphi$  &  $- \varphi$ &   $-1 $ &   1 &        0 & $- \varphi$&   $\varphi$  &         0  & $-1 $ &    0   &      0  &  1 \\
    \hline
 & $\psi$ &  $-\psi$& $-\psi$&  $\psi$ &$-1 $ & $-1 $  & 1 & 1 & 0 & 0 & 0 & 0 \\
 $ \l_3 = .4472 $ & $-1 $ &    1 &  $\psi$ &  $-\psi$ &    0  &  $\psi$ &$-\psi$ &     0   &      0   & $-1 $  &  1&    0 \\
     &    $\psi$ &   $-\psi$ &  $-1 $  &    1 & 0    & $-\psi$ &   $\psi$ &     0 & $-1 $ & 0   &0 &   1 \\
\hline
  &  $-1 $   & $-1 $ &    1   &  1   &  0   &  0    & 0   &  0   &  0  &   0   & 0   &  0\\
  &  $-1 $    & $-1 $   &   0    & 0    & 0  &   1     &1   &  0&     0   &  0    & 0  &   0\\
 $ \l_4 = -.2$ & $-1 $   & $-1 $  &   0  &   0    & 1    & 0    & 0    & 1   &  0   &  0   &  0   &  0\\
 &   $-1 $   &  $-1 $     & 0   &  0  &   0   &  0    & 0  &   0    & 0    & 1   &  1   &  0\\
  &  $-1 $   & $-1 $   &   0  &   0     &0    & 0   &  0     &0  &   1    & 0     &0  &   1 
\end{tabular}
    \caption{The rows of this matrix form an eigenbasis of the icosahedral graph. The horizontal lines divide the eigenspaces for the eigenvalues $\lambda_1,\l_2, \l_3, \lambda_4$. Here $ \varphi =  (1+\sqrt{5})/2$ and  $ \psi =  (1-\sqrt{5})/2$.  } \label{fig: icosahedron eigenbasis}
\end{figure}
Suppose we order the eigenspaces of $G$ as $\L_1 < \L_4 < \L_3 < \L_2$. Then, $G$ has no extremal combinatorial designs; Figure \ref{fig: icosahedron 2-designs reverse freq} shows the minimum cardinality positively weighted 2-designs for this ordering, which are also combinatorial. There are 12 minimum cardinality arbitrarily weighted $3$-designs, each consisting of $7$ vertices. %For the  design shown, the weight vector is $( 0.5721,0.3535,0.3535,0.3535,0.3535,0.3535,0 , 0,0,0,0, -0.2185)$  The largest positive value is the center of the cluster, and the negative value is the antipodal point. 
A minimum cardinality positively weighted $3$-design consists of 9 vertices, see Figures \ref{fig: icosa extremal}A,B. These computations were done in Matlab \cite{MATLAB:2020}.  % For the  design shown, the weight vector is $(0.2887,0.1784,0.1784,0.2887,0.4671,0.1784,0.2887,0.4671,0.4671)  The largest positive values are at the center of the cluster, and it fades outward.
 
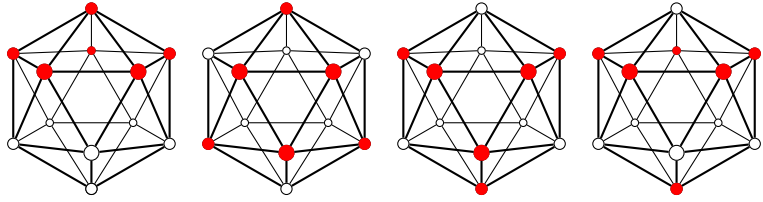
\begin{figure}[h!]
    \begin{tabular}{c c c c } \begin{tikzpicture}[ scale = .4]
  \tikzstyle{bk}=[circle,draw =black, fill=white ,inner sep=1.5pt]
      \tikzstyle{red}=[circle,draw =red, fill=red ,inner sep=1.5pt]
      \tikzstyle{pink}=[circle,draw =red, fill=red ,inner sep=1.5pt]
\tikzstyle{frontred}=[circle,draw =red, fill=red  ,inner sep=2pt]
    \tikzstyle{front}=[circle,draw =black, fill=white ,inner sep=2pt]
        \tikzstyle{backred}=[circle,draw =red, fill=red ,inner sep=1 pt]
  \tikzstyle{grey}=[fill,circle, draw=black,fill=white, inner sep=1 pt]
    \tikzstyle{blue}=[fill,circle, draw=blue!60,fill=blue!60, inner sep=1.5 pt]
    \foreach \y[count=\a] in {10,9,4}
      {\pgfmathtruncatemacro{\kn}{120*\a-90}
       \node at (\kn:3) (b\a) [bk] {} ;}
    \foreach \y[count=\a] in {8,7,2}
      {\pgfmathtruncatemacro{\kn}{120*\a-90}
       \node at (\kn:1.8) (d\a) [front] {};}
    \foreach \y[count=\a] in {1,5,6}
      {\pgfmathtruncatemacro{\jn}{120*\a-30}
       \node at (\jn:1.6) (a\a) [grey]{};}
    \foreach \y[count=\a] in {3,11,12}
      {\pgfmathtruncatemacro{\jn}{120*\a-30}
       \node at (\jn:3) (c\a) [bk] {};}
\node at (120*1 -90:3) (n1) [pink] {} ;
\node at (120*2 -90:3) (n2) [pink] {} ;
\node at (120*1-90:1.8) (n3) [frontred] {} ;
\node at (120*2 -90:1.8) (n4) [frontred] {} ;
\node at (90:1.6) (n5) [backred] {} ;
\node at (90:3) (n6) [red] {} ;
\node at (-90:3) (n7) [bk] {} ;
  \draw[grey] (a1)--(a2)--(a3)--(a1);
  \draw[thick] (d1)--(d2)--(d3)--(d1);
  \foreach \a in {1,2,3}
   {\draw[grey] (a\a)--(c\a);
   \draw[thick] (d\a)--(b\a);}
   \draw[thick] (c1)--(b1)--(c3)--(b3)--(c2)--(b2)--(c1);
   \draw[thick] (c1)--(d1)--(c3)--(d3)--(c2)--(d2)--(c1);
   \draw[grey] (b1)--(a1)--(b2)--(a2)--(b3)--(a3)--(b1);
\end{tikzpicture} 
&
    \begin{tikzpicture}[scale = .4]
  \tikzstyle{bk}=[circle,draw =black, fill=white ,inner sep=1.5pt]
      \tikzstyle{red}=[circle,draw =red, fill=red ,inner sep=1.5pt]
      \tikzstyle{pink}=[circle,draw =red, fill=red ,inner sep=1.5pt]
\tikzstyle{frontred}=[circle,draw =red, fill=red  ,inner sep=2pt]
    \tikzstyle{front}=[circle,draw =black, fill=white ,inner sep=2pt]
        \tikzstyle{backred}=[circle,draw =red, fill=red ,inner sep=1 pt]
  \tikzstyle{grey}=[fill,circle, draw=black,fill=white, inner sep=1 pt]
    \tikzstyle{blue}=[fill,circle, draw=blue!60,fill=blue!60, inner sep=1.5 pt]
    \foreach \y[count=\a] in {10,9,4}
      {\pgfmathtruncatemacro{\kn}{120*\a-90}
       \node at (\kn:3) (b\a) [bk] {} ;}
    \foreach \y[count=\a] in {8,7,2}
      {\pgfmathtruncatemacro{\kn}{120*\a-90}
       \node at (\kn:1.8) (d\a) [front] {};}
    \foreach \y[count=\a] in {1,5,6}
      {\pgfmathtruncatemacro{\jn}{120*\a-30}
       \node at (\jn:1.6) (a\a) [grey]{};}
    \foreach \y[count=\a] in {3,11,12}
      {\pgfmathtruncatemacro{\jn}{120*\a-30}
       \node at (\jn:3) (c\a) [bk] {};}
\node at (210:3) (n1) [pink] {} ;
\node at (330:3) (n2) [pink] {} ;
\node at (120*1-90:1.8) (n3) [frontred] {} ;
\node at (120*2 -90:1.8) (n4) [frontred] {} ;
\node at (90:3) (n6) [red] {} ;
\node at (-90:1.8) (n7) [frontred] {} ;
  \draw[grey] (a1)--(a2)--(a3)--(a1);
  \draw[thick] (d1)--(d2)--(d3)--(d1);
  \foreach \a in {1,2,3}
   {\draw[grey] (a\a)--(c\a);
   \draw[thick] (d\a)--(b\a);}
   \draw[thick] (c1)--(b1)--(c3)--(b3)--(c2)--(b2)--(c1);
   \draw[thick] (c1)--(d1)--(c3)--(d3)--(c2)--(d2)--(c1);
   \draw[grey] (b1)--(a1)--(b2)--(a2)--(b3)--(a3)--(b1);
\end{tikzpicture} 
&
    \begin{tikzpicture}[scale = .4]
  \tikzstyle{bk}=[circle,draw =black, fill=white ,inner sep=1.5pt]
      \tikzstyle{red}=[circle,draw =red, fill=red ,inner sep=1.5pt]
      \tikzstyle{pink}=[circle,draw =red, fill=red ,inner sep=1.5pt]
\tikzstyle{frontred}=[circle,draw =red, fill=red  ,inner sep=2pt]
    \tikzstyle{front}=[circle,draw =black, fill=white ,inner sep=2pt]
        \tikzstyle{backred}=[circle,draw =red, fill=red ,inner sep=1 pt]
  \tikzstyle{grey}=[fill,circle, draw=black,fill=white, inner sep=1 pt]
    \tikzstyle{blue}=[fill,circle, draw=blue!60,fill=blue!60, inner sep=1.5 pt]
    \foreach \y[count=\a] in {10,9,4}
      {\pgfmathtruncatemacro{\kn}{120*\a-90}
       \node at (\kn:3) (b\a) [bk] {} ;}
    \foreach \y[count=\a] in {8,7,2}
      {\pgfmathtruncatemacro{\kn}{120*\a-90}
       \node at (\kn:1.8) (d\a) [front] {};}
    \foreach \y[count=\a] in {1,5,6}
      {\pgfmathtruncatemacro{\jn}{120*\a-30}
       \node at (\jn:1.6) (a\a) [grey]{};}
    \foreach \y[count=\a] in {3,11,12}
      {\pgfmathtruncatemacro{\jn}{120*\a-30}
       \node at (\jn:3) (c\a) [bk] {};}
\node at (30:3) (n1) [pink] {} ;
\node at (270:3) (n2) [pink] {} ;
\node at (120*1-90:1.8) (n3) [frontred] {} ;
\node at (120*2 -90:1.8) (n4) [frontred] {} ;
\node at (150:3) (n6) [red] {} ;
\node at (-90:1.8) (n7) [frontred] {} ;
  \draw[grey] (a1)--(a2)--(a3)--(a1);
  \draw[thick] (d1)--(d2)--(d3)--(d1);
  \foreach \a in {1,2,3}
   {\draw[grey] (a\a)--(c\a);
   \draw[thick] (d\a)--(b\a);}
   \draw[thick] (c1)--(b1)--(c3)--(b3)--(c2)--(b2)--(c1);
   \draw[thick] (c1)--(d1)--(c3)--(d3)--(c2)--(d2)--(c1);
   \draw[grey] (b1)--(a1)--(b2)--(a2)--(b3)--(a3)--(b1);
\end{tikzpicture} 
&
  \begin{tikzpicture}[ scale = .4]
  \tikzstyle{bk}=[circle,draw =black, fill=white ,inner sep=1.5pt]
      \tikzstyle{red}=[circle,draw =red, fill=red ,inner sep=1.5pt]
      \tikzstyle{pink}=[circle,draw =red, fill=red ,inner sep=1.5pt]
\tikzstyle{frontred}=[circle,draw =red, fill=red  ,inner sep=2pt]
    \tikzstyle{front}=[circle,draw =black, fill=white ,inner sep=2pt]
        \tikzstyle{backred}=[circle,draw =red, fill=red ,inner sep=1 pt]
  \tikzstyle{grey}=[fill,circle, draw=black,fill=white, inner sep=1 pt]
    \tikzstyle{blue}=[fill,circle, draw=blue!60,fill=blue!60, inner sep=1.5 pt]
    \foreach \y[count=\a] in {10,9,4}
      {\pgfmathtruncatemacro{\kn}{120*\a-90}
       \node at (\kn:3) (b\a) [bk] {} ;}
    \foreach \y[count=\a] in {8,7,2}
      {\pgfmathtruncatemacro{\kn}{120*\a-90}
       \node at (\kn:1.8) (d\a) [front] {};}
    \foreach \y[count=\a] in {1,5,6}
      {\pgfmathtruncatemacro{\jn}{120*\a-30}
       \node at (\jn:1.6) (a\a) [grey]{};}
    \foreach \y[count=\a] in {3,11,12}
      {\pgfmathtruncatemacro{\jn}{120*\a-30}
       \node at (\jn:3) (c\a) [bk] {};}
\node at (120*1 -90:3) (n1) [pink] {} ;
\node at (120*2 -90:3) (n2) [pink] {} ;
\node at (120*1-90:1.8) (n3) [frontred] {} ;
\node at (120*2 -90:1.8) (n4) [frontred] {} ;
\node at (90:1.6) (n5) [backred] {} ;
\node at (-90:3) (n7) [red] {} ;
  \draw[grey] (a1)--(a2)--(a3)--(a1);
  \draw[thick] (d1)--(d2)--(d3)--(d1);
  \foreach \a in {1,2,3}
   {\draw[grey] (a\a)--(c\a);
   \draw[thick] (d\a)--(b\a);}
   \draw[thick] (c1)--(b1)--(c3)--(b3)--(c2)--(b2)--(c1);
   \draw[thick] (c1)--(d1)--(c3)--(d3)--(c2)--(d2)--(c1);
   \draw[grey] (b1)--(a1)--(b2)--(a2)--(b3)--(a3)--(b1);
\end{tikzpicture} 
\end{tabular}
    \caption{The isomorphism classes of the minimum cardinality positively weighted 2-design with eigenspace ordering $\L_1 < \L_4 < \L_3 <\L_2.$ These are also combinatorial designs.}
    \label{fig: icosahedron 2-designs reverse freq}
\end{figure}

However, in the ordering $\L_1 < \L_2 < \L_3 \leq \L_4$, a minimum cardinality 3-design is combinatorial and consists of only 2 vertices. Every 2-vertex 3-design in this ordering consists of a pair of antipodal points on the icosahedron (see Figure \ref{fig: icosa extremal}C). 
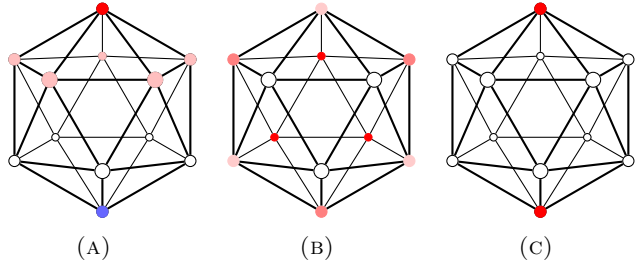
\begin{figure}[h!]

\centering
\begin{tabular}{ccc}
\subfloat[]{ 
        \begin{tikzpicture}[scale = .45]
  \tikzstyle{bk}=[circle,draw =black, fill=white ,inner sep=1.5pt]
      \tikzstyle{red}=[circle,draw =red, fill=red ,inner sep=1.5pt]
      \tikzstyle{pink}=[circle,draw =pink, fill=pink,inner sep=1.5pt]
\tikzstyle{frontred}=[circle,draw =pink, fill=pink ,inner sep=2pt]
    \tikzstyle{front}=[circle,draw =black, fill=white ,inner sep=2pt]
        \tikzstyle{backred}=[circle,draw =pink, fill=pink,inner sep=1 pt]
  \tikzstyle{grey}=[fill,circle, draw=black,fill=white, inner sep=1 pt]
    \tikzstyle{blue}=[fill,circle, draw=blue!60,fill=blue!60, inner sep=1.5 pt]
    \foreach \y[count=\a] in {10,9,4}
      {\pgfmathtruncatemacro{\kn}{120*\a-90}
       \node at (\kn:3) (b\a) [bk] {} ;}
    \foreach \y[count=\a] in {8,7,2}
      {\pgfmathtruncatemacro{\kn}{120*\a-90}
       \node at (\kn:1.8) (d\a) [front] {};}
    \foreach \y[count=\a] in {1,5,6}
      {\pgfmathtruncatemacro{\jn}{120*\a-30}
       \node at (\jn:1.6) (a\a) [grey]{};}
    \foreach \y[count=\a] in {3,11,12}
      {\pgfmathtruncatemacro{\jn}{120*\a-30}
       \node at (\jn:3) (c\a) [bk] {};}
\node at (120*1 -90:3) (n1) [pink] {} ;
\node at (120*2 -90:3) (n2) [pink] {} ;
\node at (120*1-90:1.8) (n3) [frontred] {} ;
\node at (120*2 -90:1.8) (n4) [frontred] {} ;
\node at (90:1.6) (n5) [backred] {} ;
\node at (90:3) (n6) [red] {} ;
\node at (-90:3) (n7) [blue] {} ;
  \draw[grey] (a1)--(a2)--(a3)--(a1);
  \draw[thick] (d1)--(d2)--(d3)--(d1);
  \foreach \a in {1,2,3}
   {\draw[grey] (a\a)--(c\a);
   \draw[thick] (d\a)--(b\a);}
   \draw[thick] (c1)--(b1)--(c3)--(b3)--(c2)--(b2)--(c1);
   \draw[thick] (c1)--(d1)--(c3)--(d3)--(c2)--(d2)--(c1);
   \draw[grey] (b1)--(a1)--(b2)--(a2)--(b3)--(a3)--(b1);
        \end{tikzpicture} } 
        & 
\subfloat[]{\begin{tikzpicture}[scale = .45]
    \tikzstyle{front}=[circle,draw =black, fill=white ,inner sep=2pt]
    \tikzstyle{strongest}=[circle,draw =red, fill=red ,inner sep=1pt]
    \tikzstyle{medium}=[circle,draw =red!50, fill=red!50 ,inner sep=1.5pt]
    \tikzstyle{weakest}=[circle,draw =red!20, fill=red!20 ,inner sep=1.5pt]
      \foreach \y[count=\a] in {10,9,4}
      {\pgfmathtruncatemacro{\kn}{120*\a-90}
       \node at (\kn:3) (b\a) [medium] {} ;}
    \foreach \y[count=\a] in {8,7,2}
      {\pgfmathtruncatemacro{\kn}{120*\a-90}
       \node at (\kn:1.8) (d\a) [front] {};}
    \foreach \y[count=\a] in {1,5,6}
      {\pgfmathtruncatemacro{\jn}{120*\a-30}
       \node at (\jn:1.6) (a\a) [strongest]{};}
    \foreach \y[count=\a] in {3,11,12}
      {\pgfmathtruncatemacro{\jn}{120*\a-30}
       \node at (\jn:3) (c\a) [weakest] {};}
  \draw (a1)--(a2)--(a3)--(a1);
  \draw[thick] (d1)--(d2)--(d3)--(d1);
  \foreach \a in {1,2,3}
   {\draw(a\a)--(c\a);
   \draw[thick] (d\a)--(b\a);}
   \draw[thick] (c1)--(b1)--(c3)--(b3)--(c2)--(b2)--(c1);
   \draw[thick] (c1)--(d1)--(c3)--(d3)--(c2)--(d2)--(c1);
   \draw (b1)--(a1)--(b2)--(a2)--(b3)--(a3)--(b1);
  \end{tikzpicture}} 
  &
  \subfloat[]{\begin{tikzpicture}[ scale = .45]
  \tikzstyle{bk}=[circle,draw =black, fill=white ,inner sep=1.5pt]
      \tikzstyle{red}=[circle,draw =red, fill=red ,inner sep=1.5pt]
    \tikzstyle{front}=[circle,draw =black, fill=white ,inner sep=2pt]
  \tikzstyle{grey}=[fill,circle, draw=black,fill=white, inner sep=1 pt]
    \foreach \y[count=\a] in {10,9,4}
      {\pgfmathtruncatemacro{\kn}{120*\a-90}
       \node at (\kn:3) (b\a) [bk] {} ;}
    \foreach \y[count=\a] in {8,7,2}
      {\pgfmathtruncatemacro{\kn}{120*\a-90}
       \node at (\kn:1.8) (d\a) [front] {};}
    \foreach \y[count=\a] in {1,5,6}
      {\pgfmathtruncatemacro{\jn}{120*\a-30}
       \node at (\jn:1.6) (a\a) [grey]{};}
    \foreach \y[count=\a] in {3,11,12}
      {\pgfmathtruncatemacro{\jn}{120*\a-30}
       \node at (\jn:3) (c\a) [bk] {};}
\node at (90:3) (n6) [red] {} ;
\node at (-90:3) (n7) [red] {} ;
  \draw[grey] (a1)--(a2)--(a3)--(a1);
  \draw[thick] (d1)--(d2)--(d3)--(d1);
  \foreach \a in {1,2,3}
   {\draw[grey] (a\a)--(c\a);
   \draw[thick] (d\a)--(b\a);}
   \draw[thick] (c1)--(b1)--(c3)--(b3)--(c2)--(b2)--(c1);
   \draw[thick] (c1)--(d1)--(c3)--(d3)--(c2)--(d2)--(c1);
   \draw[grey] (b1)--(a1)--(b2)--(a2)--(b3)--(a3)--(b1);
\end{tikzpicture} } 
 \end{tabular}
\caption{Three extremal designs of the icosahedral graph. Figures A and B show the isomorphism classes of the minimum weighted and positively weighted designs respectively for $\L_1 < \L_4 < \L_3 < \L_2$. Figure C shows the isomorphism class of the minimum weighted design for $\L_1 < \L_2 < \L_3 \leq \L_4$, which is combinatorial. Lighter colors correspond to weights of smaller magnitude, red indicates positive and blue is negative.}
\label{fig: icosa extremal}

\end{figure}
 %For the arbitrary weighted design shown, the weight vector is $( 0.5721,0.3535,0.3535,0.3535,0.3535,0.3535,0 , 0,0,0,0, -0.2185)$  
% For the positive weighted design shown, the weight vector is $(0.2887,0.1784,0.1784,0.2887,0.4671,0.1784,0.2887,0.4671,0.4671)  
\qed
\end{example}

Example~\ref{ex: icosa no extremal} brings up the question of eigenspace ordering. Different orderings on the eigenspaces of $G$ produce different graphical designs. A physically motivated ordering on eigenspaces is the {\em frequency ordering}, first introduced in \cite{graphdesigns}. 
The distinct eigenvalues $\l_1, \ldots, \l_m$ of $G$ are ordered by decreasing absolute value; i.e., $$1 = |\l_1| \geq |\l_2| \geq \ldots \geq |\l_n| \geq 0.$$ This induces an ordering on the eigenspaces $\L_1, \ldots, \L_m$ of $G$, though there may be ties.  In particular, the spectrum of every bipartite graph is symmetric about 0, which leads to some ambiguity. We denote $\L_i < \L_j$ when $|\l_i| > |\l_j|$, and  $\L_i \leq \L_j$ when we have chosen to break a tie by ordering $\L_i$ before $\L_j$. In Figure~\ref{fig: icosahedron eigenbasis}, the
eigenspaces are labeled by frequency. 

The frequency ordering comes from an analogy to spherical harmonics. The matrix $AD^{-1}$ is a Laplacian-type operator, which for regular graphs is in many senses analogous to the spherical Laplacian on $\SS^{n-1}$ (see \cite{HAL,BIK,singer}, for instance).  The frequency ordering on the eigenspaces of $AD^{-1}$ captures a similar notion of smoothness and symmetry that low degree spherical harmonics capture for the sphere. In this sense, graphical designs with the frequency ordering on eigenspaces extend spherical quadrature rules to the discrete domain of graphs.

% $p = new Polytope(POINTS=>[[1,-.6,-1,-.6],[1,.6,1,.6],[1,.6,-.6,-1],[1,-.6,.6,1],[1,-1,0,0],[1,-1,.6,-.6],[1,1,-.6,.6],[1,1,0,0],[1,0,0,-1],[1,0,-1,0],[1,0,1,0],[1,0,0,1]]);              has 12 vertices - is icosahedron by godsil                

% $p = new Polytope(POINTS=>[[1,1.6,-1,1.6],[1,-1.6,1,-1.6],[1,-1.6,1.6,-1],[1,1.6,-1.6,1],[1,-1,0,0],[1,-1,-1.6,1.6],[1,1,1.6,-1.6],[1,1,0,0],[1,0,0,-1],[1,0,-1,0],[1,0,1,0],[1,0,0,1]]);             
%has 10 vertices  :      (1.6,-1,1.6), (-1.6,1,-1.6 ), (-1.6,1.6,-1 ), (1.6,-1.6,1),  (-1,-1.6,1.6), (1,1.6,-1.6   ), ( 0 0 -1),(0 -1 0),(0 1 0),(0 0 1)      ? Maybe this is an error due to rounding

\section{Oriented Matroids and Eigenpolytopes} \label{sec:OM and Eigenpolytopes}

In this section we establish our main structure theorem which 
shows that there is a bijection between positively weighted $k$-designs in a 
graph and the faces of a generalized {\em eigenpolytope} of the graph. This result bestows
a great deal of combinatorial structure on $k$-designs, allowing polyhedral methods to find, organize, and optimize them. The theorem is derived via {\em Gale duality} 
from the theory of polytopes which lives under the bigger umbrella of oriented matroid
duality of vector configurations. We begin with some background.

\subsection{Oriented matroid duality of vector configurations}  \label{subsection: OM}
We introduce vector configurations and their oriented matroid duality tailored to our needs, along the lines of \cite[Chapter 6]{Ziegler}. Suppose $\mathcal{U} = \{u_1, \ldots, u_n\} \subset \RR^{n-(d+1)}$ is a collection of vectors such that the matrix $U= \begin{bmatrix} u_1 & u_2 & \cdots & u_n \end{bmatrix} \in \RR^{(n-(d+1)) \times n}$ has rank $n-(d+1)$. The {\em dual configuration} to $\mathcal{U}$ is the collection $\mathcal{U}^\ast = \{u_1^\ast, \ldots, u_n^\ast\} \subset \RR^{d+1}$ such that 
the rows of the matrix $$U^\ast = \begin{bmatrix} u_1^\ast & u_2^\ast & \cdots & u_n^\ast 
\end{bmatrix} \in \RR^{(d+1) \times n}$$ 
form a basis for the nullspace of $U$. Equivalently, $U(U^\ast)^\top = 0$ and 
$\textup{rank}(U^\ast)=d+1$. 
Assume further that $\ones$ is the first row of $U^\ast$. 
Then the following hold: 
%(i) $\mathcal{U}^\ast$ is an {\em acyclic} configuration lying entirely on one side of the hyperplane $x_0 = 0$ in $\RR^{d+1}$, 
(i) the convex hull of $\mathcal{U}^\ast$, denoted as $\conv(\mathcal{U}^\ast)$, is a polytope of dimension $d$ in $\RR^{d+1}$ lying in the hyperplane $x_1 = 1$, 
%(iii) the vector configuration $\mathcal{U}$ is {\em cyclic} in the sense that it positively spans $\RR^{n-d-1}$, 
and (ii) $\mathcal{U}$ satisfies the {\em positive dependence} relation $\sum u_i = 0$. 
A {\em dependence} on a set of vectors is a linear combination of the vectors
that equals $0$, and the dependence is {\em positive} if all coefficients in the combination are non-negative.

\begin{definition}
\begin{enumerate}
    \item A vector $c \in \RR^n$ is a {\em circuit} of the configuration $\mathcal{U}$ (or the matrix $U$) if $c \neq 0$, $Uc = 0$ and $\supp(c)$ is inclusion minimal with these properties. 
    A circuit $c$ of $\mathcal{U}$ is {\em positive} if $c \geq 0$. 
    \item  The {\em cocircuits} of $\mathcal{U}$ (or the matrix $U$) 
    are the non-zero vectors $v^\top U$ of minimal support. A cocircuit $v^\top U$ of 
    $\mathcal{U}$ is {\em positive} if $v^\top U \geq 0$. 
\end{enumerate}    
\end{definition}

The circuits of $\mathcal{U}$ are the minimal non-zero dependences of $\mathcal{U}$, or equivalently, the minimally sparse non-zero vectors in the nullspace of $U$. Similarly, the cocircuits of $\mathcal{U}$ are the minimally sparse non-zero vectors in the row space of $U$. Since the nullspace of $U$ is the row space of $U^\ast$, the circuits of $\mathcal{U}$ are precisely the cocircuits of $\mathcal{U}^\ast$ and vice-versa. 
In this sense the configurations $\mathcal{U}$ and $\mathcal{U^\ast}$ are dual. 

We refer to \cite{GrunbaumBook, Ziegler} for the basics of polyhedral theory. Here are 
the basic definitions that we will need. A {\em polytope} $P$ in $\RR^\ell$ is the convex hull of a finite set of points in $\RR^\ell$, and its dimension is the dimension of the affine span of these points. 
A (proper) {\em face} of a full-dimensional polytope $P \subset \RR^\ell$ is the intersection $P \cap \mathcal{H} \neq \emptyset$, where $\mathcal{H} \subset \RR^\ell$ is a hyperplane that contains $P$ in one of its 
closed halfspaces. A {\em facet} of $P$ is a face of $P$ of dimension $\dim(P)-1$, and a {\em vertex} of $P$ is a face of dimension $0$. The proper faces of $P$ along with $P$ and the empty set 
are all the faces of $P$, and all faces of a polytope are again polytopes. An $\ell$-dimensional polytope has at least $\ell+1$ vertices. 

There is a remarkable bijection between the positive dependences of $\mathcal{U}$ and the faces of the polytope $\conv(\mathcal{U}^\ast)$ as stated below, see \cite[p. 88]{GrunbaumBook}.

\begin{theorem}[Gale Duality] \label{thm:Gale duality}
For any $I \subseteq  [n]$,  $ \conv \{u_i^\ast: i \in [n] \setminus I\}$ is a face of $\conv(\mathcal{U}^\ast)$ if and only if $0$ is in the relative interior of $\conv\{u_i: i \in I\}$.
\end{theorem}

For $I \subseteq [n]$, $0 $ is in the relative interior of $\conv\{u_i: i \in I\}$ if and only if there is a $c \geq 0$ such that $\supp(c) = I$ and $Uc = 0$, which is if and only if $c$ is a positive dependence of $\mathcal{U}$ with $\supp(c)=I$. 
This yields the following corollary. 

\begin{corollary} \label{cor:faces and dependencies}
For any $I \subseteq  [n]$,  $ \conv \{u_i^\ast: i \in [n] \setminus I\}$ is a face (facet) of $\conv(\mathcal{U}^\ast)$ if and only if $I$ is the support of a positive dependence (circuit) of $\mathcal{U}$. 
\end{corollary}
It is customary to use the index set $J \subseteq [n]$ to denote both the collection $\{u_j \,:\, j \in J\}$ and the polytope $\conv\{u_j : i \in J\}$.

\begin{example} \label{ex:octahedron}
A classic illustration of Theorem~\ref{thm:Gale duality} is given by the dual configurations $\mathcal{U}$ and $\mathcal{U^\ast}$ shown in Figure~\ref{fig:Gale octahedron}, derived from 
\[ U =\begin{bmatrix}
    1 & 1 & -1 & -1 &  0 & 0 \\
    0 & 0 & -1 & -1 &  1 & 1
\end{bmatrix} \,\,\,\,\, \textup{ and } \,\,\,\,\,
U^\ast = \begin{bmatrix}
    1 &  1 &  1 &1 &  1 &  1 \\
    \hline
   1 & -1 & 0 & 0 & 0 & 0 \\
    0 & 0 & 1 & -1 & 0 & 0 \\
    0 & 0 & 0 & 0 & 1 & -1 
\end{bmatrix}
\] 
for which $n=6$ and $d=3$.
\begin{figure}[h!]
\begin{align*}
&
 \begin{tikzpicture}[baseline={($ (current bounding box.west) - (0,1ex) $)}, scale = .5]
\tikzstyle{bk}=[circle, fill =black ,inner sep= 2 pt,draw]
%%nodes
% nodes
\node (A) at (3.8, 0) {(1,2)};
\node (B) at (-2.6,-2.6) {(3,4)};
\node (C) at (0,3.5) {(5,6)};
\node[bk] (origin) at (0,0) {};
%%edges
\draw[thick,->]  (0,.09) -- (3, .09) ;
\draw[thick,->]  (0,-.09) -- (3, -.09) ;
\draw[thick,->]  (-.07,.07) -- (-2.17,-2.03);
\draw[thick,->]  (0.07,-.07) -- (-2.03,-2.17);
\draw[thick,->]  (.09,0) -- (.09,3) ;
\draw[thick,->]  (-.09,0) -- ( -.09,3) ;
\end{tikzpicture} &&
\begin{tikzpicture}[baseline={($ (current bounding box.west) - (0,1ex) $)}, scale = .5]
\tikzstyle{bk}=[circle, fill = white ,inner sep= 1 pt,draw]
%%nodes
\node (top) at (.7,3.5) [bk] {5};
\node (bottom) at (.7,-3.2) [bk] {6};
\node (v1) at (3.5,.5) [bk] {3};
\node (v2) at (-2,0)  [bk] {4};
\node (v3) at (2,0)  [bk] {1};
\node (v4) at (-.5,.5) [bk] {2};
%%edges
\draw (v1) --(top) -- (v2) -- (bottom) -- (v3) -- (top) -- (v4) -- (bottom) -- (v1) -- (v3) -- (v2) -- (v4) -- (v1);
\end{tikzpicture} 
\end{align*}
\caption{The configuration $\mathcal{U}$ and the octahedron $\conv(\mathcal{U}^\ast)$.} \label{fig:Gale octahedron}
\end{figure}
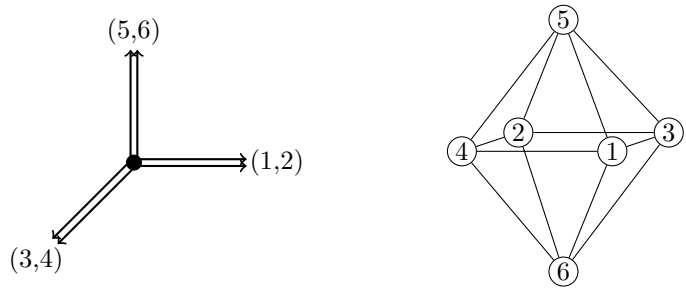
Here $\conv(\mathcal{U}^\ast)$ is an octahedron lying on the plane $x_1=1$ in $\RR^4$. 
The configuration $\mathcal{U} \subset \RR^2$ has $6$ positive circuits. The complements of their supports index the six facets of the octahedron. For example,  
$\{1,3,5\}$ is a positive circuit of $\mathcal{U}$ and its complement $\{2,4,6\}$ is a facet of $\conv(\mathcal{U}^\ast)$. The vector $(0,0,0,0,1,-1)$ is a circuit of $\mathcal{U}$ but since it is not positive, its complement $\{1,2,3,4\}$ is not a face of the octahedron; it is a cocircuit of $\mathcal{U}^\ast$ and corresponds to a hyperplane slicing through the interior of the octahedron.
\qed
\end{example}

\subsection{Graphical designs and oriented matroid duality} \label{subsec:designs and OM}
We now connect graphical designs to oriented matroid and Gale duality. 
% Let $\boldsymbol{\lambda} = \{\l_{i_1}, \ldots, \l_{i_k}\}$ denote a collection of eigenvalues, and let $\bar{\boldsymbol{\lambda}} = [m] \setminus \boldsymbol{\lambda}$. Let  $U_{\boldsymbol{\lambda}}$ denote a matrix whose rows consist of a basis for the eigenspaces $\L_{i_1}, \ldots, \L_{i_k}$. We abbreviate $\boldsymbol{\lambda} = \{\l_2, \ldots, \l_k\}$ as $\mathbf k$ and its complement  $\{\l_1, \l_{k+1}, \ldots, \l_m\}$ as $\overline{\mathbf k}$. If $\boldsymbol{\lambda} = \{\l\}$ has one element, we drop the braces and write $U_\l$.
Assume that the $m$ eigenspaces of $G$ have been ordered as $\L_1 < \cdots < \L_m$, with $\L_1 = \textup{Span}\{\ones\}$, $d_i := \dim(\L_i)$ and $s_k := \sum_{i=1}^k d_i$. Let $U \in \RR^{n \times n}$ be a matrix whose rows are a set of $n$ orthogonal eigenvectors of $G$. For any collection of eigenvalues $\boldsymbol{\l} = \{\l_{i_1}, \ldots \l_{i_j}\}$ of $G$, let $U_{\boldsymbol{\l}}$ denote the submatrix of $U$ whose rows (in order) are the eigenbases of $\Lambda_i$ for $i \in \boldsymbol{\l}$. 

Suppose we are interested in the $k$-designs of $G$ in this ordering for some $k < m$. Let $\mathbf k = \{\lambda_2, \ldots, \lambda_k\}$ and $\overline{\mathbf k} = \{\lambda_1, \lambda_{k+1}, \ldots, \lambda_m\}$. 
Then $U_{\boldsymbol{k}}$ and ${U}_{\overline{\mathbf k}}$ partition $U$ into two submatrices, each with $n$ columns, and $\ones$ is the first row of  ${U}_{\overline{\mathbf k}}$. The configuration $\cal U_{\boldsymbol{k}} \subset \RR^{s_k-1}$ is dual to the configuration $\mathcal{U}_{\overline{\mathbf k}} \subset \RR^{n-s_k+1}$ in the sense of 
Section~\ref{subsection: OM}. Keep in mind that $\cal U_{\boldsymbol{k}}$ and $\mathcal{U}_{\overline{\mathbf k}}$ are ordered vector configurations and not sets, and hence allow for repeated elements. 
%We collect several further properties into a lemma. 

\begin{lemma} 
The following hold. \label{lem: basic observations}

\begin{enumerate}
    \item $P_{\overline{\mathbf k}} := \conv(\mathcal{U}_{\overline{\mathbf k}})$ is a 
    $(n-s_k)$-dimensional polytope in $\RR^{(n-s_k)+1}$ lying on the hyperplane $x_1 = 1$. 
    \item The circuits of $\mathcal{U}_{\mathbf{k}}$ are in bijection with the cocircuits of 
    $\mathcal{U}_{\overline{\mathbf k}}$.
    \item The positive circuits of $\mathcal{U}_{\mathbf{k}}$ are in bijection with the facets of $P_{\overline{\mathbf k}}$. 
\end{enumerate}
\end{lemma}

The polytope $P_{\overline{\mathbf k}}$ is a generalized version of 
an {\em eigenpolytope} of a graph, defined by Godsil in \cite{Godsil_GGP}.

\begin{definition} \cite{Godsil_GGP} \label{def:Godsil eigenpolytope}
Let $\lambda$ be an eigenvalue of $G$, $U_\lambda$ be a matrix whose rows form a basis of the eigenspace $\L_\lambda$, and $\mathcal{U}_\lambda$ be the collection of columns of $U_\lambda$. Then the polytope $P_\lambda := \conv(\mathcal{U}_\lambda)$ is the {\em eigenpolytope} of $G$ with respect to $\lambda$. 
\end{definition}

Even though the definition of an eigenpolytope is dependent on the choice of a basis of the eigenspace,  eigenpolytopes are well defined up to combinatorial type, which is all that matters here.  Indeed, the eigenpolytopes defined from two different bases of an eigenspace 
differ by an invertible linear transformation which preserves combinatorial structure. There is a rich literature on eigenpolytopes which we will comment on at the end of this section. 

The polytopes of interest to us are of the form $P_{\overline{\mathbf k}}$ (as in Lemma~\ref{lem: basic observations}) that correspond to multiple eigenvalues of $G$ including $\lambda_1 = 1$. To address them, we generalize Definition~\ref{def:Godsil eigenpolytope} as follows. 

\begin{definition} \label{def:our eigenpolytopes}
Let $\boldsymbol{\lambda} = \{\lambda_{i_1}, \ldots, \lambda_{i_l} \}$ be a set of eigenvalues of $G$ and 
$U_{\boldsymbol{\lambda}}$ be the submatrix of $U$ consisting of the eigenvectors associated to $\boldsymbol{\lambda}$. Define $\mathcal{U}_{\boldsymbol{\lambda}}$ to be the vector configuration of the $n$ columns 
of $U_{\boldsymbol{\lambda}}$ and $P_{\boldsymbol{\lambda}} := \conv(\mathcal{U}_{\boldsymbol{\lambda}})$. The polytope $P_{\boldsymbol{\lambda}}$ is an {\em eigenpolytope} of $G$ associated to $\boldsymbol{\lambda}$. 
\end{definition}

We now have all the tools to state our main structure theorem. 

\begin{theorem} \label{thm: gale duality and designs}
Let $G = ([n],E)$ be a graph with eigenspaces ordered as 
$\L_1 < \ldots < \L_m$, let $W \subseteq [n]$ be a subset of the vertices of $G$, and let $k < m$. Consider the
dual configurations $\mathcal{U}_{\boldsymbol{k}}$ and  $\mathcal{U}_{\overline{\boldsymbol{k}}}$ as defined before. Then $W$ is a (minimal) positively weighted $k$-design of $G$ if and only if the convex hull of the subset of $\mathcal{U}_{\overline{\boldsymbol{k}}}$ indexed by $[n]\backslash W$ is a (facet) face of $P_{\overline{\boldsymbol{k}}} =\conv(\mathcal{U}_{\overline{\boldsymbol{k}}})$. 
\end{theorem}

\begin{proof}
By Lemma~\ref{lem:computational check} (2), $W$ is a (minimal) positively weighted $k$-design of $G$ if and only if $W = \supp(a)$ for a positive (circuit) dependence $a$ of $\mathcal{U}_{\boldsymbol{k}}$. The statement now follows from Corollary~\ref{cor:faces and dependencies}. 
\end{proof}

\begin{remark} \label{rem:smallest design largest facet}
The positively weighted $k$-designs of minimum cardinality correspond to the facets of $P_{\overline{\boldsymbol{k}}}$ that contain the maximum number of columns of 
${U}_{\overline{\boldsymbol{k}}}$. If all columns of ${U}_{\overline{\boldsymbol{k}}}$ are vertices of $P_{\overline{\boldsymbol{k}}}$, then the positively weighted $k$-designs of minimum cardinality correspond to the facets of $P_{\overline{\boldsymbol{k}}}$ with the maximum number of vertices. 
\end{remark}

Though it could be difficult to find the facet of a polytope containing the most vertices, we will see that this strategy is successful in examples.

\begin{example}
The {\bf Petersen graph} shown in Figure~\ref{fig:petersen} has three eigenvalues $ \l_1 = 1^{(1)}, \lambda_2 = (-2/3)^{(4)}$, and $\lambda_3 =(1/3)^{(5)}$, listed in frequency order, with multiplicity recorded as exponents. The eigenpolytopes of the Petersen graph for individual eigenvalues were studied in \cite{PadrolPfeifle} and \cite{PowersPetersen}. 

\begin{figure}[h!]
\begin{align*}
&\begin{tikzpicture}[baseline={($ (current bounding box.west) - (0,1ex) $)},scale = .5]
        \tikzstyle{bk}=[circle, fill = white,inner sep= 2 pt,draw]
%%nodes
\node (v1) at (3* .951, 3*.309) [bk] {1};
\node (v2) at (0,3)  [bk] {2};
\node (v3) at (3* -.951, 3*.309)  [bk] {3};
\node (v4) at  (3*-.588,3*-.809)  [bk] {4};
\node (v5) at (3*.588,3*-.809)   [bk] {5};
\node (w1) at (1.5* .951, 1.5*.309) [bk] {6};
\node (w2) at (0,1.5) [bk] {7};
\node (w3) at (1.5* -.951, 1.5*.309) [bk] {8};
\node (w4) at (1.5*-.588,1.5*-.809)  [bk] {9};
\node (w5) at (1.5*.588,1.5*-.809)  [bk] {10};
%% edges
\draw (w1) -- (w3) -- (w5) -- (v5) -- (v1) -- (w1) -- (w4) -- (v4) -- (v5);
\draw (v4) -- (v3) -- (w3);
\draw (v3) -- (v2) --(w2) -- (w5);
\draw (v1) -- (v2) -- (w2) -- (w4);
\end{tikzpicture}
&& U_{\overline{\mathbf{2}}} =
 \begin{bmatrix}
     1&1&1&1&1&1&1&1&1&1\\
     \hline 
      0&-1&0&1&1&0&-1&0&0&0\\
      1&-1&-1&0&1&1&-1&0&0&0\\
      0&0&1&0&0&0&-1&1&-1&0\\
      0&-1&-1&0&0&1&0&0&1&0\\
      0&-1&-1&0&1&0&0&0&0&1
      \end{bmatrix}
\end{align*} 
\caption{The Petersen graph. \label{fig:petersen}}
\end{figure}
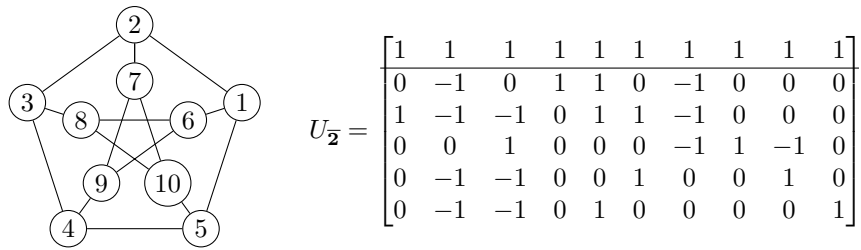
% We list the indices of the vertices in each facet of $P_{\overline{\mathbf{2}}}$ and complimentary set corresponding to minimal extremal designs in the following table.
% $$ 
% \begin{array}{|l|l||l|l|}
% \hline
% \textup{facet} & \textup{design} & \textup{facet} & \textup{design}\\
% \hline
% 1,2,3,5,8,10 & 4,6,7,9 & 1, 2, 3, 6, 8 & 4,5,7,9,10\\
% 1, 5, 6, 7, 9, 10 & 2,3,4,8 & 1, 2, 6, 7, 9 & 3,4,5,8,10\\
% 1, 3, 4, 5, 6, 8 & 2,7,9,10 & 6, 7, 8, 9, 10 & 1,2,3,4,5\\
% 2, 3, 4, 5, 7, 10 & 1,6,8,9 & 2, 3, 4, 7, 9 & 1,5,6,8,10\\
% 1, 2, 4, 5, 7, 9 & 3,6,8,10 & 3, 4, 6, 8, 9 & 1,2,5,7,10\\
% 4, 5, 6, 8, 9, 10 & 1,2,3,7 & 2, 3, 7, 8, 10 & 1,4,5,6,9\\
% 3, 4, 7, 8, 9, 10 & 1,2,5,6 & 1, 5, 6, 8, 10 & 2,3,4,7,9\\
% 1, 2, 3, 4, 6, 9 & 5,7,8,10 & 1, 2, 5, 7, 10 & 3,4,6,8,9\\
% 1, 2, 6, 7, 8, 10 & 3,4,5,9 & 1, 4, 5, 6, 9 & 2,3,7,8,10\\
% 2, 3, 6, 7, 8, 9 & 1,4,5,10 & 1,2,3,4,5,    & 6,7,8,9,10\\
% && 4,5,7,9,10 & 1,2,3,6,8\\
% && 3,4,5,8,10 & 1,2,6,7,9\\
% \hline
% \end{array} 
% $$ 
\begin{figure}[h] 
  \centering
  \begin{subfigure}[b]{0.5\linewidth} 
    \centering
    \begin{tikzpicture}[baseline={($ (current bounding box.west) - (0,1ex) $)},scale = .5]
        \tikzstyle{bk}=[circle, fill = white,inner sep= 2 pt,draw]
          \tikzstyle{red}=[circle, fill = red,inner sep= 2 pt,draw = red]
%%nodes
\node (v1) at (3* .951, 3*.309) [red] {};
\node (v2) at (0,3) [red] {};
\node (v3) at (3* -.951, 3*.309) [red] {};
\node (v4) at (3*-.588,3*-.809)  [red] {};
\node (v5) at (3*.588,3*-.809)  [red] {};
\node (w1) at (1.5* .951, 1.5*.309) [bk] {};
\node (w2) at (0,1.5) [bk] {};
\node (w3) at (1.5* -.951, 1.5*.309) [bk] {};
\node (w4) at (1.5*-.588,1.5*-.809)  [bk] {};
\node (w5) at (1.5*.588,1.5*-.809)  [bk] {};
%% edges
\draw (w1) -- (w3) -- (w5) -- (v5) -- (v1) -- (w1) -- (w4) -- (v4) -- (v5);
\draw (v4) -- (v3) -- (w3);
\draw (v3) -- (v2) --(w2) -- (w5);
\draw (v1) -- (v2) -- (w2) -- (w4);
\end{tikzpicture}
    \caption{}
  \end{subfigure}%
  \begin{subfigure}[b]{0.5\linewidth}
    \centering
  \begin{tikzpicture}[baseline={($ (current bounding box.west) - (0,1ex) $)},scale = .5]
        \tikzstyle{bk}=[circle, fill = white,inner sep= 2 pt,draw]
        \tikzstyle{red}=[circle, fill = red,inner sep= 2 pt,draw = red]
        \tikzstyle{pink}=[circle, fill = pink,inner sep= 2 pt,draw = pink]
%%nodes
\node (v1) at (3* .951, 3*.309) [bk] {};
\node (v2) at (0,3) [pink] {};
\node (v3) at (3* -.951, 3*.309) [bk] {};
\node (v4) at (3*-.588,3*-.809)  [bk] {};
\node (v5) at (3*.588,3*-.809)  [bk] {};
\node (w1) at (1.5* .951, 1.5*.309) [bk] {};
\node (w2) at (0,1.5) [red] {};
\node (w3) at (1.5* -.951, 1.5*.309) [bk] {};
\node (w4) at (1.5*-.588,1.5*-.809)  [pink] {};
\node (w5) at (1.5*.588,1.5*-.809)  [pink] {};
%% edges
\draw (w1) -- (w3) -- (w5) -- (v5) -- (v1) -- (w1) -- (w4) -- (v4) -- (v5);
\draw (v4) -- (v3) -- (w3);
\draw (v3) -- (v2) --(w2) -- (w5);
\draw (v1) -- (v2) -- (w2) -- (w4);
\end{tikzpicture}
     \caption{}
  \end{subfigure}
    \caption{The minimal positively weighted extremal designs on the Petersen graph with $\L_1 < \L_2 <\L_3$ (up to isomorphism). Figure A gives the design from a simplicial facet of $P_{\{1,3\}}$ and is combinatorial. Figure B gives the design from a 6-vertex facet and needs positive weights of two different types. \label{figure: petersen pos designs}}
\end{figure}
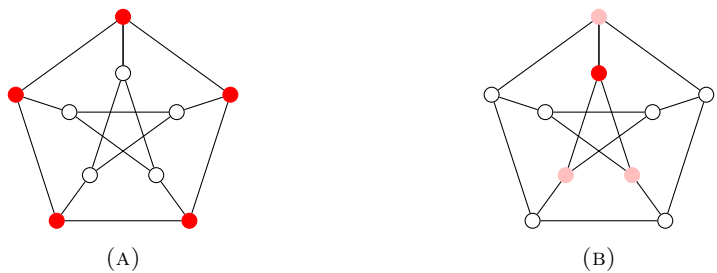

Suppose we want to know the minimal positively weighted extremal designs 
of $G$ in frequency order. By Theorem~\ref{thm: gale duality and designs}, we compute 
$\overline{\mathbf{2}} = \{1,3\}$ and the matrix $U_{\overline{\mathbf{2}}}$ whose first row is a basis for $\Lambda_1$ and next $5$ rows form a basis for $\Lambda_3$. The eigenpolytope $P_{\overline{\mathbf{2}}}$ is the convex hull of $\cal U_{\overline{\mathbf{2}}}$.
This is a $5$-dimensional polytope with $10$ vertices, $22$ facets and 
face-vector $(10, 45, 90, 75, 22)$. The $22$ facets come in two symmetry classes. There are $10$ facets with $6$ vertices and $12$ facets that are $4$-simplices. Therefore, there are $10$ minimal positively weighted extremal designs with $4=10-6$ elements and $12$ minimal positively weighted extremal designs with $5=10-5$ elements.  The two types of designs are shown in Figure \ref{figure: petersen pos designs}. 
\qed
\end{example}

Theorem~\ref{thm: gale duality and designs} provides a natural adjacency structure on the minimal positively weighted $k$-designs of $G$: two such designs are adjacent if the facets of $P_{\overline{\mathbf k}}$ that 
index them meet on a ridge (a face of codimension $2$). Therefore, the adjacency graph of  minimal positively weighted $k$-designs is precisely the edge skeleton of the polytope polar to $P_{\overline{\mathbf k}}$. 
This observation is subsumed by the general fact that the faces of $P_{\overline{\mathbf k}}$ (and hence also the faces of the polar polytope) index the 
positively weighted $k$-designs of $G$. Thus one can assign (non-minimal) positively weighted $k$-designs to the edges of the adjacency graph which we can think of as a common coarsening of the minimal designs on the two vertices of the edge. More precisely, if there are 
 two adjacent minimal positively weighted $k$-designs $W=\supp(a)$ and
 $W'=\supp(a')$, then so is 
$W \cup W'$ since $U_{\boldsymbol{k}}(a+a') = 0$ and $a+a' \geq 0$. The non-minimal design 
$W \cup W'$ lives on the edge joining $W$ and $W'$.

\begin{example} \label{ex: adjacent designs petersen}
We zoom in on the adjacency graph of the minimal positively weighted extremal designs of the Petersen graph. Figure \ref{fig: Zoomed in petersen adajcency 12345} shows the
neighborhood of the minimal design $W=\{1,2,3,4,5\}$. On the edges one sees non-minimal designs as explained above.
\end{example}

\begin{figure}[h]
    \centering
    \includegraphics[scale = .57]{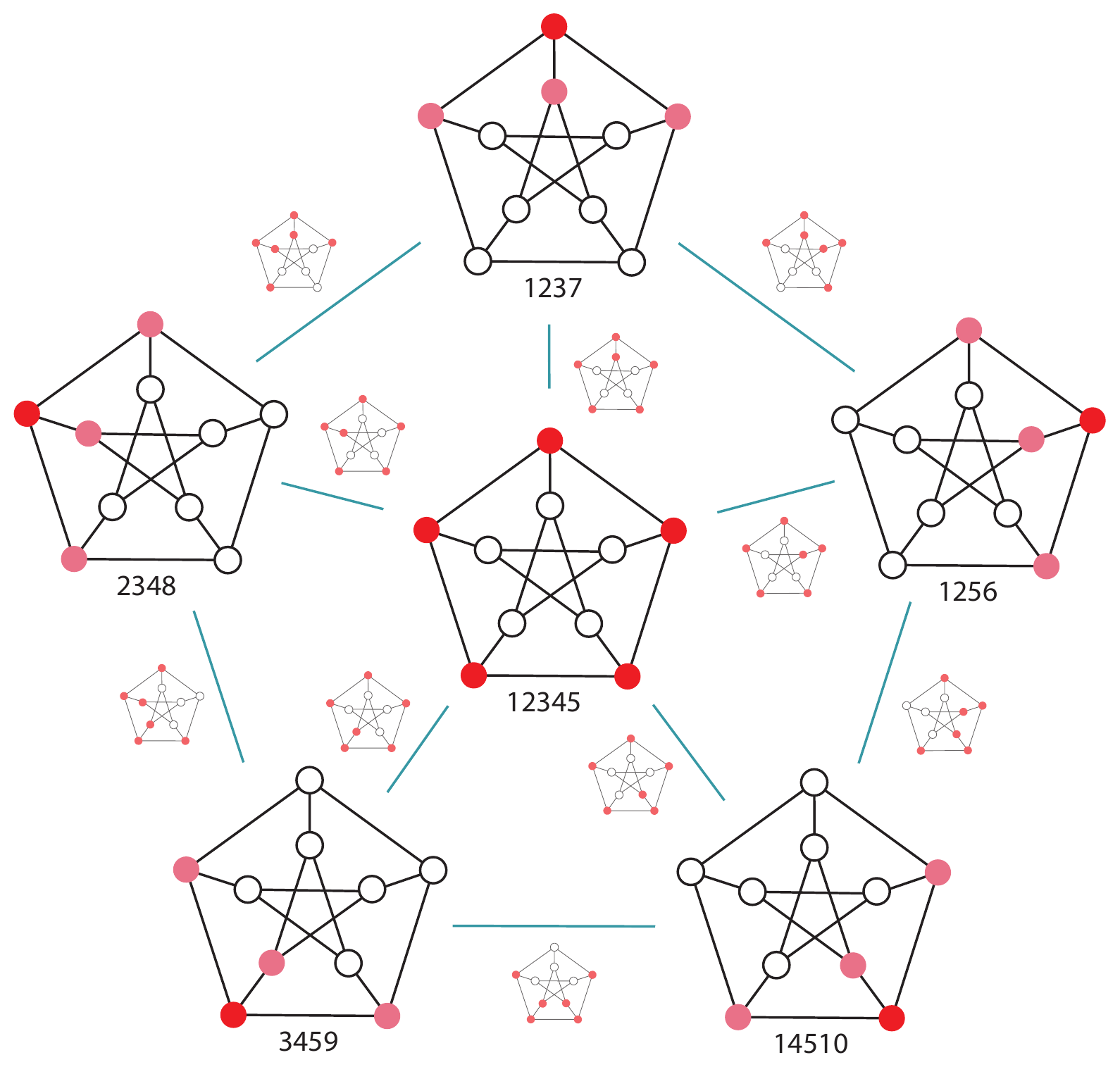}
    \caption{The minimal extremal positively weighted designs of the Petersen graph 
    adjacent to $\{1,2,3,4,5
    \}$, and the non-minimal positively weighted extremal designs connecting them. }
    \label{fig: Zoomed in petersen adajcency 12345}
\end{figure}

Next we show an example in which the eigenpolytope $P_{\boldsymbol{\l}}$ can be much simpler than what the size of $U_{\boldsymbol{\l}}$ predicts.  This happens when $U_{\boldsymbol{\l}}$ has repeated columns which can dramatically cut down the number of vertices of $P_{\boldsymbol{\l}}$. 

\begin{example} \label{ex: trunc tetra polytope}
Let $G=([12], E)$ be the {\bf truncated tetrahedral graph} (see Figure \ref{fig: trunctetra and P_5}). The eigenvalues of $G$ in frequency order are 
$$\l_1 = 1^{(1)}, \l_2 = 2/3^{(3)} ,\l_3 = -2/3^{(3)}, \l_4 = -1/3^{(3)}, \l_5 =0^{(2)}.$$ 
We look again at minimal positively weighted extremal designs in this graph which are all $4$-designs.
Recall that $\mathbf{4} = \{\l_2, \l_3,\l_4\}$ and $\overline{\mathbf{4}} = \{\l_1, \l_5\}$. We compute 
\[ U_{\bar{\mathbf{4}}} = \begin{bmatrix}
  1 & \phantom{-} 1   &    1 & 1    &    \phantom{-}1 & 1   &    1 & \phantom{-}1   &    1 & \phantom{-}1   &    1 & 1  \\
  \hline 
  1 & -1   &  0  &    0   &  -1    &  1    &  1  &   -1  &    0  &  -1   &   1 &     0\\
  0 & -1 &     1   &   1   &  -1   &   0  &    0 &    -1 &    1  &   -1&      0   &   1\\
\end{bmatrix}.\] 
Since each column of $U_{\bar{\mathbf{4}}}$ is repeated four times, $P_{\bar{\mathbf{4}}} \simeq \Delta_2$ is a triangle, shown in Figure \ref{fig: trunctetra and P_5}, with each vertex labeled by the indices of the four columns of $U_{\bar{\mathbf{4}}}$ that coincide with that vertex.  

Consider the facet (edge) of $P_{\bar{\mathbf{4}}}$ defined by the vertices with labels $\{3,4,9,12\}$ and $\{1,6,7,11\}$.  The complimentary index set is $\{2,5,8,10\}$.  By Theorem~\ref{thm: gale duality and designs}, $W = \{2,5,8,10\}$ is a minimal positively weighted $4$-design on the truncated tetrahedral graph. Moreover, there are exactly 3 such designs, one for each facet of $P_{\bar{\mathbf{4}}}$.

\begin{figure}[h]
    \centering
    \begin{align*}
&       \begin{tikzpicture}[scale =.4]
\tikzstyle{bk}=[fill = white,circle, draw=black, inner sep=2 pt]
\tikzstyle{red}=[fill =red,circle, draw=red, inner sep=2 pt]
 \foreach \y in {1,2,3,4,5,6,7,8,9}
        {\node at (40*\y +72:3) (a\y) [bk] {};}
\foreach \y in {1,2,3}
        {\node at (120*\y +30:1.2) (b\y) [bk] {};}
\node at (a8)  [red, label=right:{2}] {};
\node at (b3)  [red,label=right:{8}] {};
\node at (a3)  [red, label=left:{5}] {};
\node at (a4)  [red, label=left:{10}] {};
\draw (b1) -- (b2) -- (b3) -- (b1);
\draw (a1) -- (a2) -- (a3) -- (a4) -- (a5) -- (a6) -- (a7) -- (a8) -- (a9) --(a1);
\draw (a7) -- (a9);
\draw (a1) -- (a3);
\draw (a4) -- (a6);
\draw (b1) -- (a2);
\draw (b2) -- (a5);
\draw (b3) -- (a8);
\end{tikzpicture} 
&&    \begin{tikzpicture} 
 \tikzstyle{bk}=[circle, draw = black, fill = black ,inner sep = 2 pt]
    \node at (1,0) (a) [bk] {} ;
        \node at (1.9, 0) [] {\{1,6,7,11\}} ;
    \node at (0,1) (b) [bk] {} ;
        \node at (0,1.3) [] {\{3,4,9,12\}} ;
    \node at (-1,-1) (c) [bk] {} ;
        \node at (-1.9, .-1) [] {\{2,5,8,10\}} ;
 \draw[thick] (a)--(b)--(c)--(a);
  \fill[black, opacity =.1] (1,0) to (0,1) to (-1,-1) to (1,0);
\end{tikzpicture} 
    \end{align*}
    \caption{A minimal positively weighted 4-design and the eigenpolytope $P_{\bar{\mathbf{4}}}$ of the truncated tetrahedral graph.}
    \label{fig: trunctetra and P_5}
\end{figure}
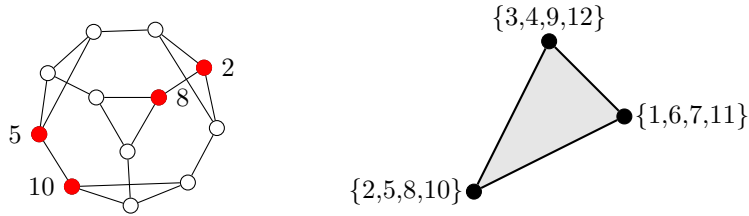

In this example, 
\[ U_{\bf 4} = \begin{bmatrix}
-1 &	-1.5 &	-0.5 &	1.5 &	2  &	1.5  &	-0.5 &	-1.5 &	-1 &	1 &	0 &	0 \\
2  &	1.5 &	1.5 &	-0.5 &	-1	 &-1.5 &-1.5 &	-0.5 &	-1 &	0 &	1	 &0 \\
-1 &	-0.5 &-1.5 &-1.5  &	-1 &	-0.5 &	1.5 &	1.5 &	2 &	0 &	0 &	1 \\
\hline 
0 &	-1 &	1 &	-1	 &0	 &1	 &-1 &	1 &	0 &	0 &	0	 &0 \\
-1 & 	0 &	1 &	-1 &	1 &	0 &	0 &	0 &	0	 &-1 &	1 &	0 \\
0 &	0 &	0 &	0 &	1 &	-1 &	1 &	0 &	-1 &	-1	 &0 &	1\\
\hline 
-1 &	0 &	1 &	0 &	-1 &	0 &	1 &	0 &	-1	 &1	 &0	 &0\\
-1 &	0 &	0 &	1 &	-1 &	0 &	0 &	1 &	-1 &	0 &	1 &	0\\
-1 &	1 &	0 &	0 &	-1 &	1 &	0 &	0 &	-1 &	0 &	0  &	1
\end{bmatrix}. \]
The circuit associated with the design $W = \{2,5,8,10\}$ is $e_2 + e_5 + e_8 + e_{10}$, a $0/1$ vector, i.e., $U_{\bf 4}\ones_W=0$. Therefore, $W$ is also combinatorial.
\qed
\end{example}

Gale duality has traditionally been used to understand the facial structure of a polytope by 
analyzing the vector configuration dual to the vertices of the polytope. The dual configuration 
is called the {\em Gale dual} of the polytope. This has been especially successful when 
the Gale dual configuration lies in a low-dimensional space, and many surprising properties of polytopes have been discovered through Gale duality (see \cite[Chapter 6]{Ziegler}, \cite{GrunbaumBook}).  In this paper we propose using Gale duality in the reverse direction, namely, use the combinatorics of eigenpolytopes to understand positively weighted graphical designs.  This is especially effective when the eigenpolytope is easy to understand. We already saw this in action in Example~\ref{ex: trunc tetra polytope} where the eigenpolytope was just a triangle. Here is a bigger example to drive home this strategy of using polytopes to 
inform designs.

\begin{example}
Let $G = ([48], E)$ be the {\bf  truncated cuboctahedral graph} shown in 
Figure~\ref{fig: truncated cubocta design}. This graph has $17$ eigenvalues, listed below in 
frequency order, with exponents showing their multiplicities:
\begin{align*}
 & 1^{(1)},-1^{(1)}, 0.9107^{(3)}, -0.9107^{(3)},
    0.7810^{(3)},
   -0.7810^{(3)},
   2/3^{(2)},
   -2/3^{(2)},
    0.6045^{(3)}, \\
  & -0.6045^{(3)},
    1/3^{(4)},
   -1/3^{(4)},
    0.2440^{(3)},
   -0.2440^{(3)},
    0.1569^{(3)},
   -0.1569^{(3)},
         0^{(4)}
\end{align*}

The matrix $U_{\overline{\mathbf{16}}}$ is made up of 8 horizontally concatenated 
copies of \[ \begin{bmatrix}
1 & 1 & 1 & 1 & 1  & 1  \\
\hline 
1 & 0 & 0 & 0 & -1 & 0  \\
0 & 1 & 0 & 0 &  0 & -1 \\
0 & 0 & 1 & 0 &  0 & -1 \\
0 & 0 & 0 & 1 & -1 & 0  \\
\end{bmatrix}.\]  
For an index set $I\subseteq [6],$ let $C(I) = \{j \in [48]: j \equiv i \mod 6  \text{ for some } i \in I \}$, that is, all vertices of $G$ indexing the columns of $U_{\overline{\mathbf{16}}}$ that correspond to columns indexed by $I$ of the above matrix. For example, $C(\{1\}) = \{ 1,7,13,19,25,31,37,43\}$.

Using Polymake \cite{Polymake}, we find that the four-dimensional polytope $P_{\overline{\mathbf{16}}}$, which has only $6$ vertices, has $8$ facets indexed by the following collections of vertices: 
\begin{align*}
    & C(\{3,4,5,6\}), C(\{2,4,5,6\}), C(\{2,3,4,5\}),   C(\{1,2,3,5\}), \\
    & C(\{1,2,3,4\}) , C(\{1,3,5,6\})   C(\{1,2,4,6\}), C(\{1,2,5,6\}).  
\end{align*}
By Theorem~\ref{thm: gale duality and designs}, $G$ has eight minimal positively weighted extremal designs which we can read off from the complements of the vertex labels of a facet.  For instance, the facet $   C(\{3,4,5,6\})$ is a certificate that 
$$C(\{1,2\}) = \{1,2,7,8,13,14,19,20,25,26,31,32,37,38,43,44\}$$
indexes a minimal positively weighted extremal design, which also happens to be a combinatorial design. We exhibit this graphical design in Figure \ref{fig: truncated cubocta design}.

\begin{figure}[h]
    \centering
    \includegraphics[scale = .5]{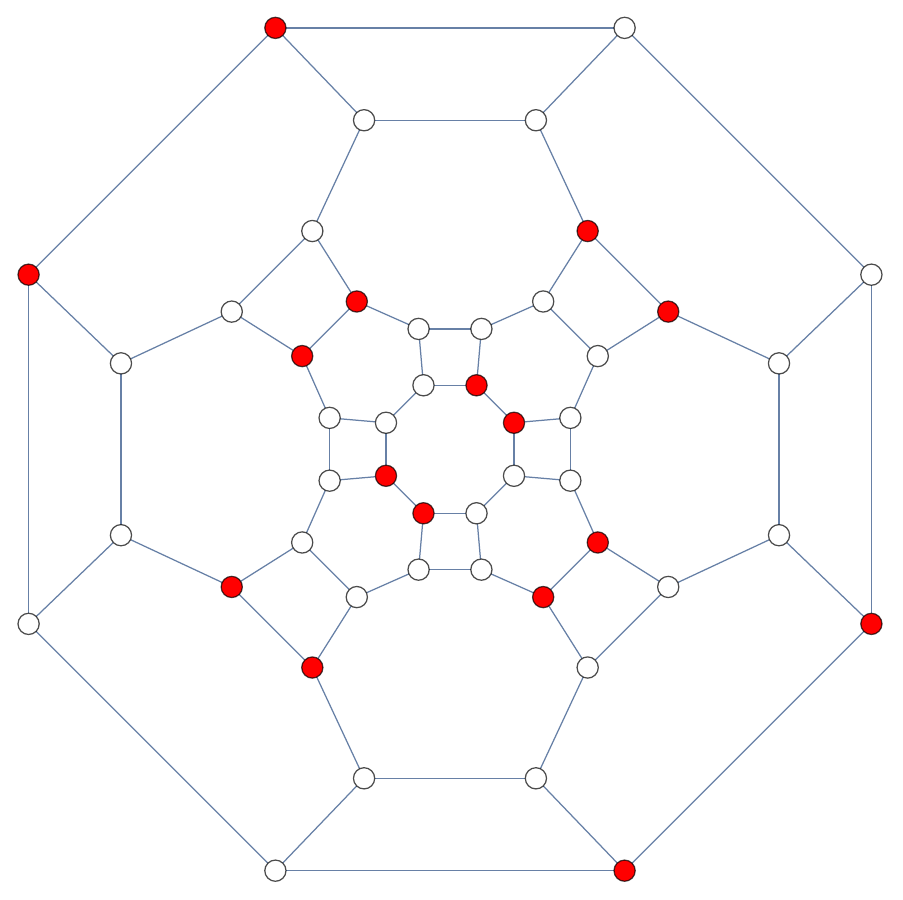}
    \caption{A minimal extremal combinatorial design on the truncated cuboctohedral graph.}
    \label{fig: truncated cubocta design}
\end{figure}

%$p = new Polytope(POINTS=>[[1,0,0,0,1],[1,0,0,1,0],[1,0,1,0,0],[1,1,0,0,0],[1,-1,0,0,-1],[1,0,-1,-1,0]]);   each point occurs *8

The facets of the polytope $P_{\overline{\mathbf{16}}}$ are easy to list and analyze, whereas a brute force enumeration of the circuits of $U_{\mathbf{16}}$ is computationally taxing.  Beginning at $t=8$, it exceeds MATLAB's \cite{MATLAB:2020} preset size restrictions to check whether each $t$-element subset of $[48]$ is a circuit. Even if we knew to look only at 16-element subsets, there are roughly $2.255\cdot 10^{12}$ of them.
\qed
\end{example}

Gale duality provides a simple upper bound on the size of positively weighted designs.  This answers an eigenspace version of  open question \#6 posed in \cite{graphdesigns}.

\begin{theorem}\label{thm:general upper bound} 
Let $G$ be a graph with eigenspaces ordered as 
$\Lambda_1 < \Lambda_2 < \cdots < \Lambda_m$, $\dim \Lambda_i = d_i$, and 
$s_k := \sum_{i=1}^k d_i$. Then 
for every $k=1, \ldots, m-1$ there is a positively weighted $k$-design of size at most $s_k$.
In particular, there is a positively weighted extremal design of size at most $s_{m-1} = n - d_m$. \end{theorem}

\begin{proof}
The eigenpolytope $P_{\bar{\bf k}}$ has dimension $\sum_{i=k+1}^m d_i = n-s_k$. Therefore any facet has at least $n-s_k$ distinct vertices and at most $n-1$ vertices. By  Theorem~\ref{thm: gale duality and designs}, the size of any minimal positively weighted $k$-design lies between $s_k$ and $1$.
\end{proof} 

\catherine{
\begin{remark}
The bijection between $k$-designs of $G$ and the faces of the 
eigenpolytope $P_{\bar{\bf k}}$ of $G$ provides 
a cursory upper bound on the maximum number of minimal positively weighted designs through the upper bound theorem for convex polytopes \cite{McMullen}. Specifically, if $P_{\bar { \bf k}}$ is $d$-dimensional, then there are at most 
$$ \binom{n - \lfloor \frac{d+1}{2} \rfloor}{n-d} +  \binom{n - \lfloor \frac{d+2}{2} \rfloor}{n-d} $$ minimal positively weighted $k$-graphical designs of $G$.
\end{remark}
}
In practice, there is often a preference for positive weights in quadrature rules. By {\em Caratheodory's theorem}
\cite{Caratheodory}, 
generically, the smallest (arbitrarily) weighted $k$-designs have size $s_k$. Theorem~\ref{thm:general upper bound} shows that for each $k$, there are {\em positively} weighted $k$-designs of size at most $s_k$, illustrating that positive weights are not too restrictive. The short proof is an immediate consequence of the Gale connection.  
There are graphs where the upper bound in Theorem~\ref{thm:general upper bound} is tight for every $k$.

\begin{example}
Consider again the icosahedral graph from Example~\ref{ex: icosa no extremal} 
with $\L_1 < \L_4 < \L_3 < \L_2$. In this 
ordering, $s_1 = 1, s_2 = 6, s_3 = 9$.  Any single vertex is a $1$-design (in any graph). We saw that a minimum cardinality positively weighted 2-design is combinatorial and has $s_2 = 6$ vertices, and a minimum cardinality positively weighted 3-design has $s_3 =9$ vertices.
\end{example}

Theorem \ref{thm:general upper bound} does not provide a general lower bound, as a facet may have $n-1$ vertices. There are indeed single vertex $k$-graphical designs when $k>1$.

\begin{example} \label{ex: Loupekines snark}
Consider the second Loupekines Snark, pictured in Figure \ref{fig: Loupekines Snark}, and label the distinguished red vertex as 1. The eigenspace $\L_2$ of this graph is 2-dimensional and is contained in $e_1^\perp$. Therefore, for $W = \{1\}$ and $\phi \in \L_2$, we have $e_1^\top \phi =0$ which means that $W = \{1\}$ averages $\L_2$. 
\begin{figure}[h]
    \centering
    \includegraphics[scale = .5]{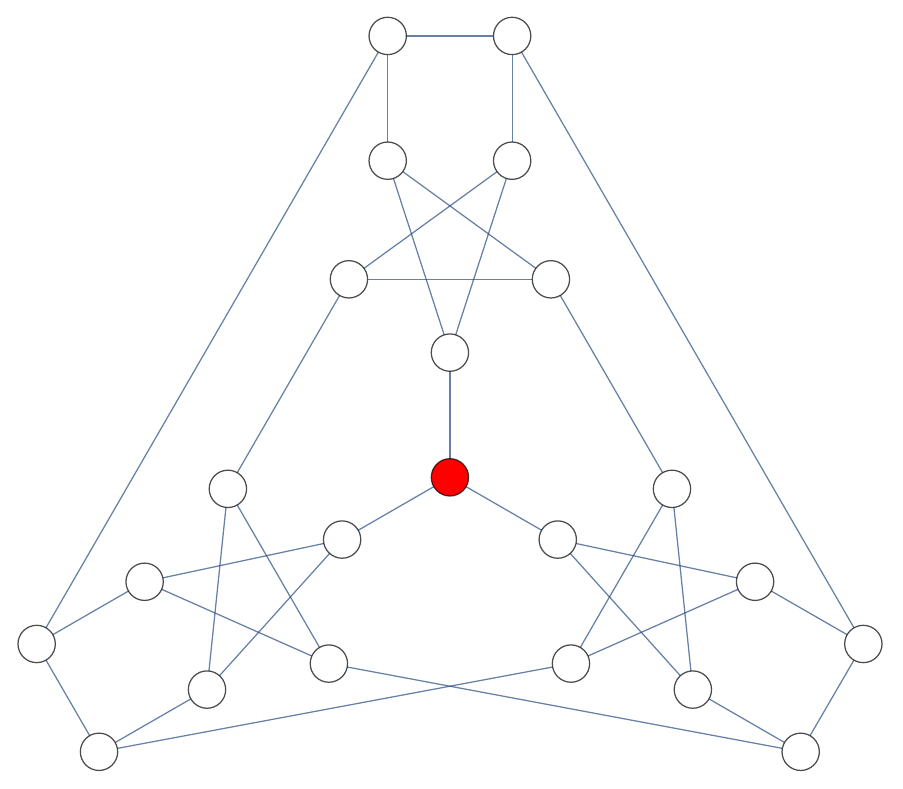}
    \caption{A combinatorial 2-graphical design on the second Loupekines Snark.}
    \label{fig: Loupekines Snark}
\end{figure}
\qed
\end{example}

We close this subsection by commenting on weighted designs that are not positively weighted.  
Lemma~\ref{lem:computational check} (1) says that $W$ is a 
weighted $k$-design of $G$ if and only if 
$W = \supp(a)$ for some dependence $a$ of $U_{\boldsymbol{k}}$ with $\ones^\top a \neq 0$.
Recall that if $U_{\boldsymbol{k}}a = 0$, then $a$ lies in the row space of 
$U_{\overline{\boldsymbol{k}}}$ and hence, $a = v_a^\top U_{\overline{\boldsymbol{k}}}$ for some $v_a$. We can think of $a$ as the vector of values of the linear functional $v_a^\top y$ on the elements of $\mathcal{U}_{\overline{\boldsymbol{k}}}$ (keeping repetitions) which contain among them the vertices of $P_{\overline{\boldsymbol{k}}}$. Let $\cal H_{v_a}$ be the hyperplane 
$${\cal H}_{v_a} := \{ y \in \RR^{n-s_k+1} \,:\, v_a^\top y = 0\}.$$
The weighted $k$-design given by $a$ is $W = \supp(a)$, and $i \in W$ if and only if 
the $i$th element of $\mathcal{U}_{\overline{\boldsymbol{k}}}$ does not lie on ${\cal H}_{v_a}$. 
If $a \geq 0$ then all of $\mathcal{U}_{\overline{\boldsymbol{k}}}$ 
lies in one halfspace of ${\cal H}_{v_a}$ and therefore also, the eigenpolytope $P_{\overline{\boldsymbol{k}}}$. If $a$ is a non-positive dependence, then there are elements
of $\mathcal{U}_{\overline{\boldsymbol{k}}}$ in both open halfspaces of ${\cal H}_{v_a}$ and 
the hyperplane ${\cal H}_{v_a}$ intersects the interior of $P_{\overline{\boldsymbol{k}}}$.
When $a$ is a circuit of $\mathcal{U}_{\boldsymbol{k}}$, the hyperplane ${\cal H}_{v_a}$ is a visualization of the corresponding cocircuit of $\mathcal{U}_{\overline{\boldsymbol{k}}}$.
We illustrate on two graphs.

\begin{example} \label{ex:weighted designs}

Consider again the icosahedral graph from Example~\ref{ex: icosa no extremal} with $\Lambda_1 < \Lambda_4 < \Lambda_3 < \Lambda_2$. By \cite[Theorem 4.3]{Godsil_DistReg}, the eigenpolytope $P_{\l_2}$ is the icosahedron again. Hyperplanes defined by circuits corresponding to minimum cardinality weighted and positively weighted $3$-designs with respect to $\L_1 < \L_4 < \L_3 < \L_2$ are shown in Figure \ref{fig: cutting planes}. The first hyperplane intersects $P_{\l_2}$ in its interior.

\begin{figure}[h!]
\begin{align*}
    & \begin{tikzpicture}[baseline={($ (current bounding box.west) - (0,1ex) $)}, scale = .5]]
  \tikzstyle{bk}=[circle,draw =black, fill=black ,inner sep=1.5pt]
      \tikzstyle{red}=[circle,draw =red, fill=red,inner sep=1.5pt]
      \tikzstyle{pink}=[circle,draw =pink, fill=pink ,inner sep=1.5pt]
\tikzstyle{frontred}=[circle,draw =pink, fill=pink ,inner sep=2pt]
    \tikzstyle{front}=[circle,draw =black, fill=black ,inner sep=2pt]
        \tikzstyle{backred}=[circle,draw =pink, fill=pink,inner sep=1 pt]
  \tikzstyle{gr}=[fill,circle, draw=black,fill=black, inner sep=1 pt]
    \tikzstyle{blue}=[fill,circle, draw=blue!60,fill=blue!60, inner sep=1.5 pt]
    \foreach \y[count=\a] in {10,9,4}
      {\pgfmathtruncatemacro{\kn}{120*\a-90}
       \node at (\kn:3) (b\a) [bk] {} ;}
    \foreach \y[count=\a] in {8,7,2}
      {\pgfmathtruncatemacro{\kn}{120*\a-90}
       \node at (\kn:1.8) (d\a) [front] {};}
    \foreach \y[count=\a] in {1,5,6}
      {\pgfmathtruncatemacro{\jn}{120*\a-30}
       \node at (\jn:1.6) (a\a) [gr]{};}
    \foreach \y[count=\a] in {3,11,12}
      {\pgfmathtruncatemacro{\jn}{120*\a-30}
       \node at (\jn:3) (c\a) [bk] {};}
\node at (120*1 -90:3) (n1) [pink] {} ;
\node at (120*2 -90:3) (n2) [pink] {} ;
\node at (120*1-90:1.8) (n3) [frontred] {} ;
\node at (120*2 -90:1.8) (n4) [frontred] {} ;
\node at (90:1.6) (n5) [backred] {} ;
\node at (90:3) (n6) [red] {} ;
\node at (-90:3) (n7) [blue] {} ;
  \draw (a1)--(a2)--(a3)--(a1);
  \draw[thick] (d1)--(d2)--(d3)--(d1);
  \foreach \a in {1,2,3}
   {\draw (a\a)--(c\a);
   \draw[thick] (d\a)--(b\a);}
   \draw[thick] (c1)--(b1)--(c3)--(b3)--(c2)--(b2)--(c1);
   \draw[thick] (c1)--(d1)--(c3)--(d3)--(c2)--(d2)--(c1);
   \draw (b1)--(a1)--(b2)--(a2)--(b3)--(a3)--(b1);
   \draw[fill = gray,opacity=0.2] (210:1.6) -- (-30:1.6) -- (-30: 3) -- (-90:1.8) -- (210:3) -- cycle;
\end{tikzpicture} 
&& 
\begin{tikzpicture}[baseline={($ (current bounding box.west) - (0,1ex) $)}, scale = .5]
    \tikzstyle{front}=[circle,draw =black, fill=black ,inner sep=2pt]
    \tikzstyle{strongest}=[circle,draw =red, fill=red ,inner sep=1pt]
    \tikzstyle{medium}=[circle,draw =red!50, fill=red!50 ,inner sep=1.5pt]
    \tikzstyle{weakest}=[circle,draw =red!20, fill=red!20 ,inner sep=1.5pt]
      \foreach \y[count=\a] in {10,9,4}
      {\pgfmathtruncatemacro{\kn}{120*\a-90}
       \node at (\kn:3) (b\a) [medium] {} ;}
    \foreach \y[count=\a] in {8,7,2}
      {\pgfmathtruncatemacro{\kn}{120*\a-90}
       \node at (\kn:1.8) (d\a) [front] {};}
    \foreach \y[count=\a] in {1,5,6}
      {\pgfmathtruncatemacro{\jn}{120*\a-30}
       \node at (\jn:1.6) (a\a) [strongest]{};}
    \foreach \y[count=\a] in {3,11,12}
      {\pgfmathtruncatemacro{\jn}{120*\a-30}
       \node at (\jn:3) (c\a) [weakest] {};}
  \draw (a1)--(a2)--(a3)--(a1);
  \draw[thick] (d1)--(d2)--(d3)--(d1);
  \foreach \a in {1,2,3}
   {\draw(a\a)--(c\a);
   \draw[thick] (d\a)--(b\a);}
   \draw[thick] (c1)--(b1)--(c3)--(b3)--(c2)--(b2)--(c1);
   \draw[thick] (c1)--(d1)--(c3)--(d3)--(c2)--(d2)--(c1);
   \draw (b1)--(a1)--(b2)--(a2)--(b3)--(a3)--(b1);
    \fill[gray,opacity=0.2] (30:1.8) -- (150:1.8) -- (-90:1.8) -- cycle;
  \end{tikzpicture}
  &&&
  \begin{tikzpicture}[baseline={($ (current bounding box.west) - (0,1ex) $)}, scale = .45]
\tikzstyle{bk}=[circle, fill = black,inner sep= 2 pt,draw, thick]
\tikzstyle{red}=[circle, fill = red,inner sep= 2 pt,draw=red, thick]
\tikzstyle{blue}=[circle, fill = blue,inner sep= 2 pt,draw = blue, thick]
%%nodes
\node (top) at (.7,3.5) [red] {};
\node (bottom) at (.7,-3.2) [blue] {};
\node (v1) at (3.5,.5) [bk] {};
\node (v2) at (-2,0)  [bk] {};
\node (v3) at (2,0)  [bk] {};
\node (v4) at (-.5,.5) [bk] {};
%%edges
\draw[thick] (v1) --(top) -- (v2) -- (bottom) -- (v3) -- (top) -- (v4) -- (bottom) -- (v1) -- (v3) -- (v2) -- (v4) -- (v1);
\fill[gray,opacity=.2] (3.5,.5) -- (2,0)  -- (-2,0) --(-.5,.5) -- cycle;
\end{tikzpicture}
\end{align*}
    \caption{Hyperplanes corresponding to circuits.}
    \label{fig: cutting planes}
\end{figure}
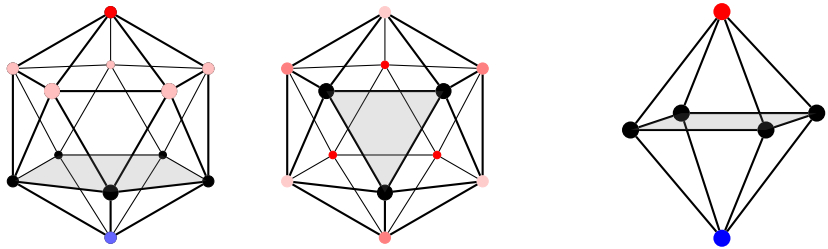

One has to be careful when dealing with non-positively weighted designs since in that case, a dependence $a$ gives rise to a graphical design if and only if $\ones^\top a \neq 0$.  Consider the graph of the regular octahedron in Figure~\ref{fig: cutting planes}. By \cite[Theorem 4.3]{Godsil_DistReg} again, the eigenspace of $\l  = 0$ is an octahedron. From  Example~\ref{ex:octahedron}, $a=(0,0,0,0,1,-1)$ is a circuit and the corresponding hyperplane bisects the octahedron as shown. However $\ones^\top a = 0$, and so this circuit $a$ fails to average $\Lambda_1$.

\qed
\end{example}

\subsection{Eigenpolytope Literature} \label{subsec: Eigenpolytope Lit}
Eigenpolytopes of a graph $G$ with respect to single eigenvalues were defined by Godsil in \cite{Godsil_GGP}. He used the symmetries of eigenpolytopes to understand $\textup{Aut}(G)$, the automorphism group of $G$.  Our eigenpolytopes in Definition~\ref{def:our eigenpolytopes} are more general in that they involve sets of eigenvalues, and using them to study graphical designs is new. 

Most of the existing work on eigenpolytopes focuses on the second largest eigenvalue of the adjacency matrix $A$, which they term $\theta_1$.  The eigenpolytope $P_{\theta_1}$ is often closely tied to the structure and symmetry of $G$;  most notably for distance regular graphs \cite{Godsil_DistReg, Powers}. We also refer the reader to further work on symmetry and families of graphs in \cite{ChanGodsil, RooneyThesis,PowersPetersen}. Padrol and Pfeifle \cite{PadrolPfeifle} translate graph operations to operations on eigenpolytopes arising from the graph Laplacian $D-A$. They mention in passing that eigenpolytopes might connect to oriented matroids and Gale duality. For regular graphs, the eigenspaces of $A$, $AD^{-1}$, $D-A$, and many other common graph matrices are equivalent.

The $n$ columns of the matrix $U_\lambda$ whose rows form a basis for the eigenspace of $\lambda$ 
can be thought of as an embedding of the $n$ vertices of $G$. In some situations the 
edge graph of the eigenpolytope $P_\lambda = \conv(\mathcal{U}_\lambda)$ is isomorphic to $G$ 
and provides a {\em spectral drawing} of $G$, see \cite{WinterPaper,WinterThesis,GodsilGeomOfDistReg,PowersLicata} for more on this connection. McConnell in \cite{McConnell} computes an extensive list of spectral graph drawings of the graphs of uniform polyhedra. Eigenpolytopes corresponding to multiple eigenvalues are, in spirit, analogous to spectral drawings using eigenvectors of multiple eigenvalues.

%$p = new Polytope(POINTS=>[[1,-1,-1,-1,-1,-1],[1,1,0,0,0,0],[1,0,1,0,0,0],[1,0,0,1,0,0],[1,0,0,0,1,0],[1,0,0,0,0,1]]);
% P_4 each point occurs *2 

%$p = new Polytope(POINTS=>[[1,-.6,-1,-.6],[1,.6,1,.6],[1,.6,-.6,-1],[1,-.6,.6,1],[1,-1,0,0],[1,-1,.6,-.6],[1,1,-.6,.6],[1,1,0,0],[1,0,0,-1],[1,0,-1,0],[1,0,1,0],[1,0,0,1]]);    
% P_2 

% $p = new Polytope(POINTS=>[[1,1.6180,-1,1.6180,-1,-1,-1,-1,-1],[1,-1.6180,1,-1.6180,-1,-1,-1,-1,-1,],[1,-1.6180,1.6180,-1,1,0,0,0,0],[1,1.6180,-1.6180,1,1,0,0,0,0],[1,-1,0,0,0,0,1,0,0],[1,-1,1.6180,-1.6180,0,1,0,0,0],[1,1,-1.6180,1.6180,0,1,0,0,0],[1,1,0,0,0,0,1,0,0],[1,0,0,-1,0,0,0,0,1],[1,0,-1,0,0,0,0,1,0],[1,0,1,0,0,0,0,1,0],[1,0,0,1,0,0,0,0,1]]); 
% P_24

\subsection{Connections to Extremal Combinatorics} \label{subsec:Golubev}
The extremal designs found by Golubev in \cite{Golubev} provide classes of faces of certain eigenpolytopes.  We rephrase his results in our language. 
We first recall the {\em Hoffman bound},\catherine{ which does not actually appear in the standard reference \cite{Hoffman}. The origins and generalizations of this theorem are explained in \cite{haemer}. If $G$ is strongly regular, the linear programming bound of \cite{DelsarteThesis} is the Hoffman bound.}

\begin{theorem}(Hoffman Bound \cite{haemer})
Let $G$ be a regular graph on $n$ vertices, let $\l_{\min}$ be the least eigenvalue of $ A D^{-1}$, and let $\a(G)$ be the size of a maximum stable set of $G$. 
Then,
$$\frac{\a(G)}{n} \leq \frac{-\l_{\min}}{1-\l_{\min}}.$$
\end{theorem}

The following theorem is a translation of \cite[Theorem 2.2]{Golubev} to eigenpolytopes.

\begin{theorem}
Let $G$ be a regular graph for which the Hoffman bound is sharp.  Then, a maximum stable set and its complement each provide a face of $P_{\l_{\min}}$.
\end{theorem}

\begin{proof}
Let $G= (V,E)$ have the eigenspace of $\l_{\min}$ ordered last.  By \cite[Theorem 2.2]{Golubev} and \cite[Theorem 5.6]{cubescodes}, a maximum stable set $W\subseteq V$ is an extremal combinatorial design, and $V\setminus W$ is an extremal combinatorial design as well by Lemma \ref{lem: comb complements}. By Gale duality, it follows that $W$ and $V\setminus W$ index faces of $P_{\l_{\min}}$.
\end{proof}

A {\em cut} in a connected 
graph $G = (V,E)$ is a partition of the 
vertex set $V$ into a set $W \subset V$ and its complement. An edge $ij \in E$ is a {\em cut edge} in the cut induced by $W$ if one endpoint is in $W$ and the other in $V\setminus W$. We use $E(W, V \setminus W)$ to denote the set of cut edges in the cut induced by $W$. 
The second main result of \cite{Golubev} relies on the following variant of the Cheeger bound.

\begin{theorem}(\cite{AM,Tanner})
\emph{Let $G$ be a connected $\d$-regular graph and $\theta_1$ be the second largest eigenvalue of $ A D^{-1}$. Then} 
$$ \underset{\vn \neq W \subsetneq V}{\min}\frac{|V| |E(W, V\setminus W)|}{\d|W||V\setminus W|} \geq 1-\theta_1 .$$ 
\end{theorem}

\begin{theorem}
Let $G$ be a graph for which the Cheeger bound is sharp.  Then, a set which realizes the Cheeger bound and its complement each provide a face of $P_{\theta_1}$.
\end{theorem}
\begin{proof}
Let $G= (V,E)$ have the eigenspace of $\theta_1$ ordered last.  By \cite[Theorem 2.4]{Golubev} and \cite[Theorem 5.11]{cubescodes}, a set $W\subset V$ realizing the Cheeger bound is an extremal combinatorial design, and hence so is $V\setminus W$ by Lemma \ref{lem: comb complements}. By Gale duality, it follows that $W$ and $V\setminus W$ index faces of $P_{\theta_1}$.
\end{proof}

\section{Cocktail Party Graphs and Cycles}
\label{sec:cross-polytopes and cycles}

We now use the power of Gale duality via Theorem~\ref{thm: gale duality and designs} to investigate three families of graphs. 
Since they are all Cayley graphs, we begin with general results about their spectrum 
which is closely tied to the representation theory of finite groups (\cite{SerreBook},\cite[Chapter 1]{Sagan}). 

Given a group $H$ and a generating set $S \subseteq H$  such that $s\in S \implies s^{-1} \in S$, one can construct the connected, $|S|$-regular  Cayley graph $\Gamma(H,S)$. The vertices are indexed by $H$, and there is an edge between group elements $g$ and $h$ if $g = sh$ for some $s\in S$. The eigenvectors of $\Gamma(H,S)$ are given by the group characters of $H$. When $H$ is a finite abelian group, this provides a simple method to compute the spectrum of $\Gamma(H,S)$ since there are $|H|$ one-dimensional representations of $H$. The eigenvalue of an eigenvector $\phi$ (a group character of $H$) can be computed from:
$$(AD^{-1} \phi)(g) = \frac{1}{|S|} \sum_{h \sim g}  \phi(hg)  = \left( \frac{1}{|S|} \sum_{h \sim g}  \phi(h) \right)\phi(g) .$$
The eigenvectors depend only on the group $H$, but the eigenvalues, and hence groupings of eigenvectors into eigenspaces, depends on the generating set $S$ as well. 

\subsection{The Cocktail Party Graph}

We denote the regular {\em cross-polytope} in dimension $d$ by $\Diamond_d = \conv\{\pm\, e_i : i \in [d]\}$. Its edge graph is commonly known as the {\em cocktail party graph} and is a regular 
graph with $2d$ vertices and degree $2(d-1)$. This graph can also be defined as the {\em complete multipartite graph} $K_{2,\ldots,2}$, the Cayley graph $\Gamma(\ZZ_{2d}, [2d-2])$, and the complement of $d$ disjoint copies of the path $P_2$ with two vertices.  In frequency order, the spectrum of the cocktail party graph is 
$$1^{(1)},\left(\frac{-1}{d-1}\right)^{(d-1)}, 0^{(d)};$$ 
see \cite{BrouwerHaemers} for a reference. It is quick to compute the matrix $U$; 
we label columns by the vertices of $\Diamond_d $, and the row blocks by eigenvalues: 

\begin{center}
\begin{tabular}{c||c c c c c c c }
& & & & vertex\\
 & $e_1$  &  $-e_1$ &    $\cdots$   &     $\cdots$ &    $\cdots$    &    $e_d$ & $-e_d $  \\
\hline \hline
$\l_1 = 1$ &&&&$\ones^\top$ &&\\
    \hline
 $ \l_2 = -(d-1)^{-1}$ &  $e_1$  &  $e_1$ &   $\cdots$    &    $e_{d-1}$ & $e_{d-1}  $ & $-\ones$ &     $-\ones$\\
\hline
 $ \l_3 = 0$ & $e_1$  &  $-e_1$ &    $\cdots$   &     $\cdots$ &    $\cdots$    &    $e_d$ & $-e_d $  
\end{tabular}
\end{center}

\begin{theorem} In frequency order, the graph of $\Diamond_d $ has $2^d$ minimal positively weighted extremal designs, each of which consists of $d$ vertices and corresponds to a facet of $\Diamond_d$. Each of these designs is combinatorial.
\end{theorem}

\begin{proof} 
By \cite[Theorem 4.3]{Godsil_DistReg}, the eigenpolytope for $\l = 0$ is again the cross-polytope $\Diamond_d$, and hence the extremal eigenpolytope $P_{\overline{\boldsymbol{2}}}$ is isomorphic to $\Diamond_d$. The cross-polytope is simplicial, which is to say that every facet is a $d$-simplex $\Delta_{d-1}$, which has exactly $d$ vertices.  Moreover, a subset of $d$ vertices is a facet of $\Diamond_d$ if and only if its complement is also a facet. 
Hence, by Theorem~\ref{thm: gale duality and designs},  the minimal positively weighted extremal designs of the cocktail party graph consist of $2d-d = d$ vertices, and correspond to the $2^d$ facets of $\Diamond_d$. Examining the eigenspace $\L_2$, it is quick to see that the weight vector for each such design is a 0-1 vector, hence these are combinatorial designs.
\end{proof}

\begin{theorem} \label{thm:not frequency order}
With $\Lambda_2$ ordered last, the graph of $\Diamond_d $ has $d$ minimal positively weighted extremal designs, each of which consists of $2$ antipodal vertices of $\Diamond_d$ and is combinatorial. 
\end{theorem}

\begin{proof}
The eigenpolytope $P_{\l_2}$ is $\Delta_{d-1}$ with the antipodal vertices of $\Diamond_d$ 
collapsed into a single vertex of $\Delta_{d-1}$. The complements of facets of $\Delta_{d-1}$ are exactly the vertices of $\Delta_{d-1}$.  By Theorem~\ref{thm: gale duality and designs}, 
there are $d$ minimal positively weighted extremal designs of $\Diamond_d$ in this ordering of eigenspaces, each of which is a pair of antipodal vertices. Examining the eigenspace $\L_3$, it is clear that the weight vector for each such design is of the form $e_i + e_{i+1}\in \RR^{2d}$, hence these are combinatorial designs.
\end{proof}

Allowing arbitrary weights, the minimal extremal designs 
of the cocktail party graph are combinatorial, and in particular, positively weighted. 

\begin{theorem} \label{thm: crosspolytope only pos weights}
Every minimal weighted extremal design of the cocktail party graph is combinatorial. 
\end{theorem}

\begin{proof}
In frequency order, it follows from the structure of the eigenbasis of $\Lambda_2$ that if we allow arbitrary weights, then any circuit of $U_2$ (up to sign) is of the form $a \in \RR^{V(\Diamond_d)}$ with 
$a_{e_{j}} = 1$, $a_{-e_{j}}=-1$ for some $j$ and $0$ in all other coordinates. 
Therefore  $\ones^\top a = 0$, and $\supp (a)$ is not a graphical design by 
Lemma~\ref{lem:computational check} (1). This generalizes the observation in Example~\ref{ex:weighted designs}. For the other ordering, the hyperplanes $\cal H$ representing cocircuits of $\mathcal{U}_2$ are precisely the spans of the facets of the eigenpolytope $P_{\lambda_2} = \Delta_d$ which gives combinatorial designs as in Theorem~\ref{thm:not frequency order}.
\end{proof}

\subsection{Cycles}

We next consider the cycle graph $C_n= \Gamma(\ZZ_n, \{\pm 1\})$. Let 
\[ \phi_j :=\begin{bmatrix}
1\\
\cos(2\pi j/n)\\
\cos(2\pi 2j/n)\\
\vdots \\
\cos(2\pi (n-1)j/n)
\end{bmatrix}, 
\psi_j :=
\begin{bmatrix}
0\\
\sin(2\pi j/n)\\
\sin(2\pi 2j/n)\\
\vdots \\
\sin(2\pi (n-1)j/n)
\end{bmatrix}.\]
For each $ j \in [\lfloor (n-1)/2 \rfloor]$, $\phi_j$ and $\psi_j$ span a two-dimensional eigenspace $\L_j$ of $C_n$ with eigenvalue $\l_j = \cos(2 \pi j/n)$.  If $n$ is even, then $\phi_{n/2}$ and $\psi_{n/2}$ collapse into the one-dimensional eigenspace $\L_{n/2} = \spanset \{ ((-1)^t): t \in [n]\}$. Therefore, all eigenpolytopes of $C_n$ corresponding to a single eigenvalue (other than $\l_1=1$) are either a polygon or a line segment. We note that the above indexing of the eigenspaces is not compatible with frequency ordering.  In frequency order, the highest frequency eigenspace(s) will correspond to $j \approx n/4$, since $\cos(2 \pi j/ n) \approx \cos(\pi /2)=0$. We consider minimal extremal designs of $C_n$ in frequency order and will see that their 
structure depends on the congruence class of $n \mod 4$. 

\begin{theorem} \label{thm: 0 mod 4 cycles}
Let $n \equiv 0 \mod 4.$ In frequency ordering, $C_n$ has four minimal positively weighted extremal designs, each consisting of $n/2$ vertices. They are 
\newline $\{ i, i+1 \in [n]: i \equiv j \mod 4\} $ for $j \in [4]$.
 \end{theorem}
 
 \begin{proof}
 If $n \equiv 0 \mod 4$, then the extremal eigenspace $\L_{n/4}$ with eigenvalue $0$ is spanned by
 \[ \phi_{n/4} =\begin{bmatrix}
1 \\
\cos(\pi /2)\\
\cos(\pi)\\
\vdots \\
\cos(\pi (n-1)/2)
\end{bmatrix} 
\,\,\,\, \textup{ and } \,\,\,\,
\psi_{n/4}=
\begin{bmatrix}
0\\
\sin(\pi /2)\\
\sin(\pi)\\
\vdots \\
\sin(\pi(n-1) /2)
\end{bmatrix}.\]
Each element of $\cal U_{n/4}$ is one of the following:
\[ \begin{bmatrix} 
1 \\
0
\end{bmatrix},
\begin{bmatrix} 
0\\ 
1
\end{bmatrix},
\begin{bmatrix} 
-1\\ 
0
\end{bmatrix},
\begin{bmatrix} 
0 \\ 
-1
\end{bmatrix}.\]
Thus $P_{n/4}$ is the diamond $\Diamond_2$ with $n/4$ graph vertices indexing each polytope vertex. Each facet is indexed by the vertices $\{ i, i+1 \in [n]: i \equiv j \mod 4\} $, and each facet is also the complement of a facet, Thus there are $4$ minimal positively weighted extremal designs of the stated form.
 \end{proof}

 \begin{corollary}  \label{cor: 0 mod 4 cycles only positive}
Every minimal extremal design of $C_n$, for $n \equiv 0 \mod 4$, is positively weighted. 
\end{corollary}
 
 \begin{proof}
 The extremal eigenpolytope of $C_n$ is $\Diamond_2$.  By Theorem \ref{thm: crosspolytope only pos weights}, the only cocircuit hyperplanes $\cal H$ that yield designs are the 
 those that support facets of $\Diamond_2$. 
 \end{proof}

 \begin{figure}[h]
    \begin{tabular}{c c c c} 
    \rotatebox{67.5}{\begin{tikzpicture}[ scale = .4]
        \tikzstyle{bk}=[circle, draw = black, fill = white ,inner sep=2.5pt]
        \tikzstyle{red}=[circle, draw = red, fill = red ,inner sep=2.5pt]
    \foreach \y[count=\a] in {1,2,3,4,5,6,7,8}
      {\pgfmathtruncatemacro{\kn}{45*\a - 45}
       \node at (\kn:3) (b\a) [bk] {} ;}
       \node at (0:3) (b1) [red] {} ;
       \node at (45:3) (b2) [red] {} ;
       \node at (180:3) (b5) [red] {} ;
       \node at (225:3) (b6) [red] {} ;
  \draw (b1)--(b2)--(b3)--(b4)--(b5)--(b6) --(b7)--(b8)--(b1);
\end{tikzpicture} }
& 
   \rotatebox{30}{ \begin{tikzpicture}[ scale = .4]
        \tikzstyle{bk}=[circle, draw = black, fill = white ,inner sep=2.5pt]
        \tikzstyle{red}=[circle, draw = red, fill = red ,inner sep=2.5pt]
        \tikzstyle{light}=[circle,draw =red!20, fill=red!20 ,inner sep=2.5pt]
        \tikzstyle{med}=[circle,draw =red!50, fill=red!50 ,inner sep=2.5pt]
         \tikzstyle{medplus}=[circle,draw =red!60, fill=red!60 ,inner sep=2.5pt]
       \node at (0:3) (b9) [light] {} ;
       \node at (40:3) (b1) [med] {} ;
       \node at (80:3) (b2) [med] {} ;
       \node at (120:3) (b3) [light] {} ;
       \node at (160:3) (b4) [bk] {} ;
       \node at (200:3) (b5) [medplus] {} ;
       \node at (240:3) (b6) [red] {} ;
       \node at (280:3) (b7) [medplus] {} ;
       \node at (320:3) (b8) [bk] {} ;
  \draw (b1)--(b2)--(b3)--(b4)--(b5)--(b6)--(b7)--(b8)--(b9)--(b1);
\end{tikzpicture} }
& 
    \rotatebox{54}{\begin{tikzpicture}[ scale = .4]
  \tikzstyle{bk}=[circle, draw = black, fill = white ,inner sep=2.5pt]
        \tikzstyle{red}=[circle, draw = red, fill = red ,inner sep=2.5pt]
        \tikzstyle{light}=[circle,draw =red!50, fill=red!50 ,inner sep=2.5pt]
       \node at (0:3) (b10) [light] {} ;
       \node at (36:3) (b1) [red] {} ;
       \node at (72:3) (b2) [light] {} ;
       \node at (108:3) (b3) [bk] {} ;
       \node at (144:3) (b4) [bk] {} ;
       \node at (180:3) (b5) [light] {} ;
       \node at (216:3) (b6) [red] {} ;
       \node at (252:3) (b7) [light] {} ;
       \node at (288:3) (b8) [bk] {} ;
       \node at (324:3) (b9) [bk] {} ;
  \draw (b1)--(b2)--(b3)--(b4)--(b5)--(b6)--(b7)--(b8)--(b9)--(b10)--(b1);
\end{tikzpicture} }
\end{tabular}
\vspace{-.2 in}
    \caption{Minimal positively weighted extremal designs of $C_8,C_9,C_{10}$.}
    \label{fig: extremal cycles }
\end{figure}
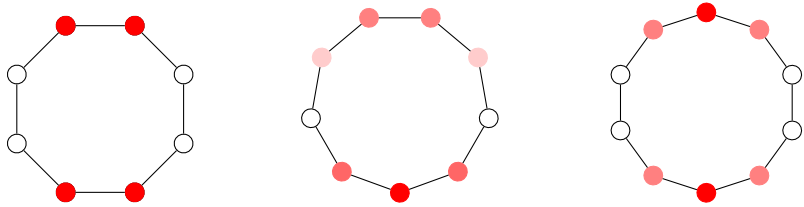
For the remaining types of cycles, we need some short gcd calculations. 
\begin{lemma} \label{lem: gcd} We have the following calculations.
\begin{enumerate}
    \item If $n \equiv 1 \mod 4$, then $\gcd(n, (n - 1)/4) = 1$, $\gcd(2n, (n - 1)/2) = 2$, and $\gcd(2n, (n + 1)/2) = 1$. 
    \item If  $n \equiv 3 \mod 4$, then $\gcd(n, (n+1)/4) = 1$, $\gcd(2n, (n + 1)/2) = 2$, and $\gcd(2n, (n - 1)/2) = 1$. 
\end{enumerate}
\end{lemma}
 
\begin{proof}
\begin{enumerate}
    \item Let $n = 4k+1$. Then, $k = (n-1)/4$, so
\begin{align*}
    \gcd(n, (n - 1)/4)  =   \gcd(k , 4k+1 ) =   \gcd(k , 1)  =1.
\end{align*}
Since $(n+1)/2 = 2k+1$ is odd,
\begin{align*}
    \gcd(2n, (n+1)/2)  =   \gcd(8k+2 , 2k+1 ) =    \gcd(2k+1,2k-1) = \gcd(2 , 2k-1)  =1.
\end{align*}
Similarly,  $(n-1)/2 = 2k$ is even, so
\begin{align*}
    \gcd(2n, (n-1)/2)  =   \gcd(8k+2 , 2k ) =    \gcd(2,2k) =2.
\end{align*}

\item This follows similarly by writing $n = 4k-1$.
\end{enumerate}
\end{proof}

 \begin{theorem}  \label{thm: odd cycles}
 If $n$ is odd, then every minimal positively weighted extremal design of $C_n$ 
 consists of $n-2$ vertices. 
 \end{theorem}
 
 \begin{proof}
 If $n \equiv 1 \mod 4$, the extremal eigenspace $\L_{(n-1)/4}$ with eigenvalue $\cos(\pi (n-1)/4n)$ is spanned by
 \[ \phi_{(n-1)/4}(v) =  \cos(2\pi v {(n-1)/4}n) \quad \textup{ and } \quad
\psi_{(n-1)/4}(v) = \sin(2 \pi v {(n-1)/4}n).
\]
By Lemma \ref{lem: gcd}, $\gcd (n, (n - 1)/4) = 1.$  Thus the values $ \cos(2\pi v (n-1)/4n)$ are distinct and $\| ( \cos(2\pi v {(n-1)/4}n), \sin(2 \pi v {(n-1)/4}n)) \|_2 =1$ as $v$ ranges over $[n]$. Therefore, 
%and so the vertices  of $P_{(n-1)/4}$ are all distinct. Furthermore, , so no vertices are co-linear. Therefore, 
$P_{(n-1)/4}$ is an $n$-gon, and every facet contains exactly two vertices. The statement about graphical designs then follows from Theorem~\ref{thm: gale duality and designs}.
The case of $n \equiv 3 \mod 4$ with extremal eigenspace $\L_{(n+1)/4}$ follows similarly.
\end{proof}

 \begin{theorem}  \label{thm: 2 mod 4 cycles}
 If $n \equiv 2 \mod 4$, then every minimal extremal design of $C_n$ consists of $n-4$ or $n-2$ vertices depending on which of the eigenspaces indexed by $(n\pm 2)/4$ is ordered last.
  \end{theorem}
 \begin{proof}
 
Let $n \equiv 2 \mod 4$.  There is a tie for the final eigenspace of $C_n$ in the frequency order, so we consider both $(n\pm 2)/4$.  We claim that the extremal eigenpolytope $P_m$ for $m = (n\pm 2)/4$ is the $n$-gon if $m$ is odd, and is the $(n/2)$-gon doubled up if $m$ is even. 
 Let $n = 2j$, where $j$ is odd.  Then $m = (j \pm 1)/2$, and $\L_{m}$  is spanned by
 \begin{align*}
     &\phi_{m}(v) =  \cos(2 \pi v m/n) =  \cos\left(2 \pi v \frac{j\pm 1}{2}/2j\right)  \textup{ and }\\
&\psi_{m}(v) = \sin(2 \pi v m/n) = \sin\left(2 \pi v \frac{j\pm 1}{2}/2j\right)  .
 \end{align*} 
Suppose $m$ is odd.  Then by Lemma \ref{lem: gcd}, $\gcd(m,n) = 1$, and so the columns of ${U}_m$ are distinct as $v$ ranges over $[n]$.  If $m$ is even, then $\gcd(m,n) = 2$ by Lemma \ref{lem: gcd}. So, every column of $U_m$ occurs with multiplicity 2, and regardless of the parity of $m$, every column lies on the unit circle. This proves the claim about $P_m$. 

The doubled $(n/2)$-gon has $n/2$ facets each containing 4 vertices, and the $n$-gon has $n$ facets each containing 2 vertices.  By Theorem~\ref{thm: gale duality and designs}, every minimal  positively weighted extremal design of $C_n$ then consist of $n-4$ or $n-2$ vertices depending on which of the eigenspaces indexed by $(n\pm 2)/4$ is ordered last.
\end{proof}

\begin{theorem}
The minimum cardinality of an arbitrarily weighted extremal design in $C_n$ is achieved by positively weighted extremal designs.
\end{theorem}

\begin{proof}
As noted in several proofs, the vertices of an extremal eigenpolytope $P_m$ lie on the unit circle, and so at most two distinct vertices can lie on a line. This is the same number of distinct vertices on each facet of $P_m$.  
\end{proof}

\section{Graphs of Hypercubes} 
\label{sec:cubes}

As our final example, we consider extremal designs in the edge graphs of cubes. 
Let $Q_d = \Gamma(\{0,1\}^d, \{e_1, \ldots e_d\})$ denote the edge graph of the $d$-dimensional hypercube $[0,1]^d$, which is $d$-regular.
The Hamming weight of a vector $x \in \{0,1\}^d$ is $|x| := \ones^\top x$, the number of $1$'s in the coordinates of $x$. 
For $i=0,\ldots,d$, let  
$$\J_{d,i} = \{ x \in \{0,1\}^d \,:\, |x|=i\}$$ be the 
collection of ${d \choose i}$ vertices of $Q_d$ with Hamming weight $i$. We refer to 
$\J_{d,i}$ as the $i$-th slice of the (Boolean) $d$-cube. 
It is convenient to index the spectrum of $Q_d$ by the slices of the cube, which differs 
from frequency ordering.
\begin{itemize}
    \item The eigenvalues of $Q_d$ are $\lambda_i = 1-2i/d$ for $i=0, \ldots, d$.
    \item The eigenvalue $\lambda_i$ has multiplicity ${d \choose i} = |\J_{d,i}|$. 
    \item An eigenbasis of the eigenspace of $\lambda_i = 1-2i/d$ is given by the 
    vectors $$\left\{ \phi_x(y) = ((-1)^{x^\top y} \,:\, y \in \{0,1\}^d)^\top \,:\,  x \in \J_{d,i} \right\}.$$
\end{itemize}
 The last eigenvalue/eigenspace in frequency order corresponds to the middle level of the cube, $i \approx d/2$. Let $P_i$ denote the eigenpolytope of $\lambda_i$.
 
\subsection{Eigenpolytopes of the hypercube}
For a fixed $i \in \{0,1,\ldots,d\}$, the matrix $U_i$ has rows indexed by the ${d \choose i}$ vectors $x \in \J_{d,i}$, columns by each $y \in \{0,1\}^d$, with $(x,y)$-entry equal to $(-1)^{x^\top y}$. Let the column indexed by $y$ be denoted as $c(y)$. The eigenpolytope $P_i$ is a full-dimensional polytope in $ \RR^{{d \choose i}}$, and every element of $\mathcal{U}_i$ is a vertex of $P_i$ since they are all $\pm 1$ vectors. 

The matrices $U_0$ and $U_d$ have one row each. Since $U_0 = [ \ones ]$, $P_0 = (\Delta_0)^{2^d}$ is a $0$-simplex, i.e., a point, labeled by all elements of 
$\{0,1\}^d$. The exponent denotes the multiplicity of labels on a vertex of the polytope. 
The unique row of $U_d$ has entries $c(y)= (-1)^{\ones^\top y}$ which records the parity of $|y|$. Therefore $P_d = (\Delta_1)^{2^{d-1}}$ is a line segment with half the elements of $\{0,1\}^d$ labeling each endpoint. Leaving out these eigenpolytopes, we can describe the others.

\begin{lemma} \label{lem:odd and even eigenpolytopes} 
Fix $i \in \{1, \ldots, d-1\}$. 
\begin{enumerate}
    \item If $i$ is odd, then the columns of $U_i$ come in pairs of oppositely signed vectors; $c(y) = -c(\ones - y)$ for all $y \in \{0,1\}^d$. The eigenpolytope $P_i$ is a centrally symmetric polytope of dimension ${d \choose i}$ with $2^d$ vertices. 
    \item If $i$ is even, then the columns of $U_i$ come in pairs of equal vectors; 
$c(y) = c(\ones - y)$ for all $y \in \{0,1\}^d$. The eigenpolytope $P_i$ has dimension ${d \choose i}$ and $2^{d-1}$ vertices. Each vertex has two labels given by the two columns of $U_i$ that give that vertex. 
\end{enumerate}
\end{lemma}    

\begin{proof}
All of the columns of $U_i$ are vertices of $P_i$ since they are all $\pm 1$ vectors. 
However, it could be that two or more columns correspond to the same vertex of $P_i$. Consider the pair of vectors $y \in \{0,1\}^d$ and $\ones -y \in \{0,1\}^d$. For $x \in \J_{d,i}$, if $x^\top y = \alpha$, then $x^\top(\ones - y) = |x| - \alpha$. Therefore, if $i = |x|$ is odd then $\alpha$ and $|x|-\alpha$ have opposite parity and if $i = |x|$ is even then $\alpha$ and $|x|-\alpha$ have the same parity. So if $i$ is even, $c(y) = c(\ones - y)$ and $y$ and $\ones -y$ index the same vertex of $P_i$. If $i$ is odd, $c(y) = -c(\ones - y)$, and $P_i$ is centrally symmetric. 

To finish the proof, we need to argue that there are no further identifications 
of columns in $U_i$. 
%We now show that when $i$ is even, each vertex of $P_i$ is labeled by no more than two columns of $U_i$.
Consider the matrix \[ A= \begin{bmatrix}
 x_1 &
 \cdots &
   x_{\binom{d}{i}}
\end{bmatrix}^\top \]
whose rows are made up of the vectors $x_j \in \cal J_{d,i}$, and partition $\{0,1\}^d$ as $\mathcal{Y}\sqcup \mathcal{Y}'$ so that $y \in \mathcal{Y}$, $\ones - y \in \mathcal{Y}'$ and for any $y \in \mathcal{Y}, $ $|y| \leq d/2$. By the above, $c(\mathcal{Y}) = \pm c(\mathcal{Y}')$ depending on the parity of 
$i$. So, it suffices to show that the map $\phi \,:\, \mathcal{Y} \rightarrow \{0,1\}^{d \choose i}$ given by $\phi(y) = Ay \textup{ mod } 2$ is injective. That is, we want to show there are no linear dependences of $A$ coming from $\cal Y$. Suppose $Ay = 0$ and $|y| \leq d/2$. Let $I = \supp(y)$.   If $|I| \geq  1$, we can find $x_j$ so that $| \supp (x_j) \cap I |$ is odd, since every vector with weight $i$ appears as a row of $A$.  Therefore $x_j^\top y =1$, a contradiction, and $\phi$ is injective. 
\end{proof}

\begin{example} 
Consider $Q_3$ the edge graph of the $3$-cube. The matrix 
$$U = \left[ \begin{array}{rrrrrrrr}
1 & 1 & 1 & 1 & 1 & 1 & 1 & 1 \\
\hline
1 & -1 & 1 & 1 & -1 & -1 & 1 & -1 \\
1 & 1 & -1 & 1 & -1 & 1 & -1 & -1 \\
1 & 1 & 1 & -1 & 1 & -1 & -1 & -1 \\
\hline 
1 & -1 & -1 & 1 & 1 & -1 & -1 & 1 \\
1 & -1 & 1 & -1 & -1 & 1 & -1 & 1 \\
1 & 1 & -1 & -1 & -1 & -1 & 1 & 1 \\
\hline
1 & -1 & -1 & -1 & 1 & 1 & 1 & -1 
\end{array} \right]$$
with successive blocks indexed by the eigenvalues ordered by Hamming weight:
$$(\lambda_0 = 1)^1, \left(\lambda_1 = \frac{1}{3}\right)^3, \left(\lambda_2 = -\frac{1}{3}\right)^3, (\lambda_3 = -1)^1.$$
Here are the non-trivial eigenpolytopes of $Q_3$:
\begin{itemize}
\item $i = 1$: $P_1$ is a centrally symmetric $3$-polytope with $8$ vertices. It is in fact, a 
regular $3$-cube by \cite[Theorem 4.3]{Godsil_DistReg}, which has $6$ facets with $4$ vertices each. Therefore the smallest positively weighted designs that average all eigenspaces, except the one indexed by $\lambda_1$, have size $4$.

\item $i = 2$: $ P_2 = \textup{conv} \{ (1,1,1), (-1,-1,1), (-1,1,-1), (1,-1,-1)\} = (\Delta_3)^2$
is a tetrahedron. Each vertex has two labels. Counting this multiplicity of labels, the maximum number of vertices on a facet of $P_2$ is $6$ which implies that the smallest positively weighted 
designs that average all eigenspaces, except the one indexed by $\lambda_2$, have size $2$.
\end{itemize}

\end{example}

\subsection{The Cut Polytope}

Consider the cut in a graph $G=([d],E)$ induced by a subset $S$ of vertices (c.f. Subsection~\ref{subsec:Golubev}.)
The {\em incidence vector} of the cut is $\chi_S \in \{0,1\}^E$ such that 
$(\chi_S)_{ij} = 1$ if $ij \in E(S,[d]\setminus S)$, and $0$ otherwise. 
%The MAX-CUT problem in $G$ is to find a cut in $G$ that maximizes $|E(S,[d]\setminus S)|$.  
% MAX-CUT is known to be NP-complete from Karp's classic list of 21 NP-complete problems in \cite{Karp}. 
 Let $K_d$ be the complete graph on $d$ vertices. The {\em cut polytope} of $K_d$, denoted $\textup{CUT}_d^\square$, is the convex hull of all incidence vectors of cuts in $K_d$.
%If $\ones_G \in \{0,1\}^E$ is the indicator vector of $G$ with a $1$ in position $ij$ if and only if $ij \in E$, then MAX-CUT in $G$ is solved by the linear program: 
%\begin{equation}
%    \min \ones_G^\top x \,:\, x \in \textup{CUT}_d^\square.
%\end{equation}
The following is possibly well-known and appears in \cite{PadrolPfeifle} without proof.

\begin{lemma} \label{lemma: cut polytope}
For the graph $Q_d$, the eigenpolytope $P_{2} $ is isomorphic to $ \textup{CUT}_d^{\square}$, the cut polytope of the complete graph $K_d$ on $d$ vertices. 
\end{lemma}

\begin{proof}
The $(x,y)$-entry of $U_2$ is $c(y)_x=(-1)^{x^\top y} = (-1)^{|\supp (x)\ \cap \ \supp (y)|}$, for $x \in \cal J_{d,2}$, and $y \in \{0,1\}^d$. The vectors $x$ are precisely the incidence vectors of all edges in $K_d$, and the vectors $y \in \{0,1\}^d$ are the incidence vectors of all possible subsets of $[d]$. Every subset of $[d]$ induces a cut in $K_d$. We now
show that $c(y)_{ij} = -1$ if $ij$ is in the cut induced by $\supp(y)$, and $c(y)_x =1$ otherwise. 

Let $y \in \{0,1\}^d$.  If $x$ is an edge of $K_d$ in the cut induced by $\supp(y)$, then $|\supp (x) \cap \supp (y)| =1$, and so $(-1)^{x^\top y}= -1$.  If $x$ is not a cut edge, then either  $|\supp (x) \cap \supp (y)| =2$ if $x$ indexes an edge contained in $\supp(y)$, or $|\supp (x )\cap \supp (y )|=0$ if $x$ indexes an edge contained in the complement.  Either way, $(-1)^{x^\top y} = 1$. Thus $\cal U_2$ consists of 
$\pm 1$ vectors indexing cuts in $K_d$. 

We have defined $\cut$ as the convex hull of $0/1$ vectors indexing the cuts of $K_d$. Note that $P_2$ is the image of $\cut$ under the map $x \mapsto -2x+\ones$.
\end{proof}
% \catherine{
% \begin{theorem}
% Let $G = (V,E)$ be a graph with some ordering of its eigenspaces: $\L_1, \ldots, \L_m$. There exists no polynomially concise description of the minimal positively weighted extremal designs of $G$ unless NP = co-NP.
% \end{theorem}

% \begin{proof}
%  As noted in \cite[p.50]{DezaLaurentBook}, a result of Karp and Papadimitriou \cite{KarpPapadimitriou} implies that there exists no polynomially concise facet description of $\textup{CUT}_d^\square$ unless NP = co-NP.
% Consider the graph $Q_d$ with the eigenspace $\L_2$ ordered last.  By Lemma \ref{lemma: cut polytope}, the eigenpolytope $P_2$ is $\textup{CUT}_d^\square$. By Theorem~\ref{thm: gale duality and designs}, the minimal positively weighted extremal designs on $Q_d$ are in bijection with the facets of $\cut$. 
% \end{proof}}

Since  $\chi_S = \chi_{[d] \setminus S}$, the columns of $U_2$ come in pairs of identical vectors and $P_2\simeq \cut$ has $2^{d-1}$ vertices, each 
corresponding to 2 vectors in $\mathcal{U}_2$.  
Although the facet structure of $\cut$ is notoriously difficult to understand, its facets with the maximum number of vertices are well understood.

\begin{proposition}[Prop 26.3.12 of \cite{DezaLaurentBook}] \label{prop: DL max facets}
For any facet $F$ of $\cut$, 
$$ |\{\chi_S \in F \}| \leq 3 \cdot 2^{d-3},$$
with equality if and only if $F$ is defined by a triangle inequality. 
\end{proposition}

The {\em triangle inequalities} are among the simplest facet inequalities of $\cut$, each of the form
\begin{align*}
    x_{ij} - x_{ik} - x_{jk} \leq 0 \,\,\, \text{ or } \,\,\,
    x_{ij} + x_{ik} + x_{jk} \leq  2
\end{align*}
for distinct $i,j,k \in [d]$. By Theorem~\ref{thm: gale duality and designs} and Proposition~\ref{prop: DL max facets}, the facets of $\cut$ from triangle inequalities are in bijection with the 
minimum cardinality extremal designs of the graph $Q_d$ with the eigenspace $\L_2$ ordered last. 

% \begin{lemma} \label{lem: graph auts}
% Let $f \in \Aut (G)$ be a graph automorphism.  If $W = \supp (a)$ is a $k$-graphical design with weight vector $a$, then the subset $f(W) = \{f(w) : w\in W\}$ is a $k$-graphical design with weight vector $f(a)$, where $f(a)_v = a_{f(v)}$. 
% \end{lemma}

% \begin{proof}
% The automorphism $f$ is given by a permutation matrix $P$ which commutes with $A$ and by extension any polynomial in $A$.  The matrix $U_{\bf k}^\top U_{\bf k} $ is the orthogonal projector onto $\spanset \{ \L_2, \ldots, \L_k \}$, known to be polynomial in $A$. Since $U_{\bf k} a = 0$,
% $$ U_{\bf k}^\top U_{\bf k} f(a) = U_{\bf k}^\top U_{\bf k} P a = P  U_{\bf k}^\top U_{\bf k} a = 0.$$
% Because $U_{\bf k} U_{\bf k}^\top = I$, multiplying on the left by $U_{\bf k} $ shows that $U_{\bf k} f(a) =0$.
% Additionally, we see that $\ones^\top f(a) = \ones^\top a  = 1$.
% \end{proof}

\begin{theorem} \label{thm:cut polytope min designs}
Suppose we order the eigenspaces of $Q_d$, for $d \geq 3$,  so that $\Lambda_2$ is last. Then there are $4 \binom{d}{3} $ minimum cardinality extremal designs of $Q_d$, each consisting of $2^{d-2}$ vertices. 
\end{theorem}

\begin{proof}

The extremal eigenpolytope in this situation is $P_2 \cong \cut$ and hence it suffices to reason about the facts of $\cut$. Every $\cut$ vertex $\chi_S$  corresponds to the pair of $Q_d$ vertices, $\ones_S$ and $\ones_{[d]\setminus S}$ indexed by a subset $S$ of $[d]$. 
There are $4 \binom{d}{3}$ facets of $\cut$ described by triangle inequalities. By Proposition \ref{prop: DL max facets}, these are precisely the facets of $\cut$ with the maximum number of $3 \cdot 2^{d-3}$ vertices.

Fix three distinct indices $i,j,k \in [d]$ and consider the set of 
vertices of $Q_d$ 
\[W = \{ \ones_S, \ones_{[d] \setminus S}: i,j,k \in S\}. \] 
Note that there are $2^{d-3}$ subsets of $[d]$ containing 
$i,j,k$. 
For $\ones_S \in W$ and a facet $F$ of $\cut$ defined by the triangle inequality 
$$x_{ij} + x_{ik} + x_{jk} \leq  2,$$
$\chi_S \not \in F$ since none of the edges $ij, ik$, or $jk$ are cut: $(\chi_S)_{ij} + (\chi_S)_{ik} + (\chi_S)_{jk} = 0 <2 $. 
Of the $2^{d-1}$ vertices of $\cut$, $F$ contains $3 \cdot 2^{d-3}$ vertices and $W$ accounts for $2^{d-3}$ vertices not on $F$. Therefore, 
$\ones_S \in W $ with $i,j,k \in S$ if and only if $ \chi_S \notin F$. 
By Theorem \ref{thm: gale duality and designs}, $W$ is a minimum extremal graphical design on $Q_d$ and it has cardinality $2 \cdot 2^{d-3} = 2^{d-2}$.

Likewise, for a facet $F$ defined by the triangle inequality 
$$x_{ij} - x_{ik} - x_{jk} \leq 0,$$
  $\chi_S \notin F$ if and only if $i,j \in S $ and $k \not\in S$. 
  By a similar argument to the one above, the design corresponding to this facet has minimum cardinality and is 
 $$ W' = \{ \ones_S, \ones_{[d] \setminus S}: i,j \in S, k \not\in S\} .$$ 
\end{proof}

\subsection{$Q_d$ and the Frequency Order}
In the rest of this section we focus on the extremal eigenpolytope (and designs) of $Q_d$ in 
frequency order. 
In this order, the last eigenvalue of $Q_d$ is $\lambda_{\frac{d}{2}} = 0$ when $d$ is even and $\lambda_{\frac{d+1}{2}} = -1/d$ or $\lambda_{\frac{d-1}{2}} = 1/d$ when $d$ is odd. These correspond to the ``middle'' slice(s) of the $d$-cube. 
Let $m$ (for ``middle'') denote the index of the last eigenspace of $Q_d$ in frequency order. Lemma~\ref{lem:odd and even eigenpolytopes} says the following about $P_m$.

\begin{lemma} \label{lem:maximal eigenpolytopes of cubes}
Depending on the parity of $d$ and further, the congruence class of $d$ mod $4$, the extremal eigenpolytope of $Q_d$ in frequency order is the following: 
\medskip 

\begin{enumerate}
    \item  $d$ even: Then $m = d/2$ and $P_m$ is a polytope of dimension ${d \choose \frac{d}{2}}$. 
    \begin{enumerate}
        \item $d \equiv 0$ mod $4$: Then $P_m$ is a polytope with $2^{d-1}$ vertices. Each vertex of $P_m$ comes from two identical columns of $U_m$ given by $c(y) = c(\ones -y)$. 
        \item $d \equiv 2$ mod $4$: Then $P_m$ is a centrally symmetric polytope with $2^d$ vertices. 
        For each vertex $c(y)$ of $P_m$, $-c(y) = c(\ones -y)$ is also a vertex of $P_m$. 
    \end{enumerate}
    
    \bigskip 
    
    \item $d$ odd: Then there are two possible indices $m = (d+1)/2$ and $m'= (d-1)/2$ for the extremal eigenpolytopes, and  $\dim(P_m) = {d \choose \frac{d+1}{2}} = {d \choose \frac{d-1}{2}} = \dim(P_{m'})$.
    \begin{enumerate}
        \item $d \equiv 1$ mod $4$: In this case, $m :=  (d+1)/2$ is odd while $m' :=  (d-1)/2$ is even. 
        Therefore, $P_{m}$ is centrally symmetric with $2^d$ vertices while $P_{m'}$ has $2^{d-1}$ vertices. 
        \item $d \equiv 3$ mod $4$: In this case $m := (d+1)/2$ is even while $m' :=  (d-1)/2$ is odd. Therefore, $P_m$ has $2^{d-1}$ vertices while $P_{m'}$ is centrally symmetric and has $2^d$ vertices. 
    \end{enumerate}
\end{enumerate}
\end{lemma}

Lemma~\ref{lem:maximal eigenpolytopes of cubes} allows us to upper bound the size of a minimal positively weighted extremal design in $Q_d$ using Theorem~\ref{thm: gale duality and designs}.

\begin{corollary}\label{cor:upper bounds for cubes}
Let $W$ be a minimal positively weighted extremal design in $Q_d$. 
\begin{enumerate}
        \item If $d \equiv 0$ mod $4$, then $|W| \leq 2^d - 2 {d \choose \frac{d}{2}}$. 
        \item If $d \equiv 1$ mod $4$, then 
        $$|W| \leq   \min \left\{2^d - 2 {d \choose \frac{d-1}{2}}, 2^d - {d \choose \frac{d+1}{2}} \right\} = 2^d - 2 {d \choose \frac{d-1}{2}}.$$ 
        \item If $d \equiv 2$ mod $4$, then 
        $|W| \leq 2^d - {d\choose \frac{d}{2}}$.  
        \item If $d \equiv 3$ mod $4$, then 
        $$|W| \leq \min \left\{ 2^d - 2{d \choose \frac{d+1}{2}}, 2^d - {d \choose \frac{d-1}{2}} \right\} = 2^d - 2{d \choose \frac{d+1}{2}}.$$ 
\end{enumerate}
\end{corollary}

Except for $d \equiv 2$ mod $4$, the upper bounds given in Corollary~\ref{cor:upper bounds for cubes} are strictly better than those in Theorem~\ref{thm:general upper bound} because of the doubling of columns in the matrix $U_m$. 
    %These upper bounds are $d - 2d_m$ as opposed to $d-d_m$ in Theorem~\ref{thm:general upper bound}. 
    Even so, the bounds in Corollary~\ref{cor:upper bounds for cubes} are not tight, and we will 
    improve them in the next subsection. For example in $Q_5$,  Corollary~\ref{cor:upper bounds for cubes} says that there is a positively weighted extremal design of size at most $12$ while in fact there is one of size $8$. This is because the eigenpolytope $P_2$ of $Q_5$ has facets with $12$ vertices, each given by two columns of 
    $U_2$, and so there is an extremal design of size $32-24 = 8$. Since $\dim(P_2) = 10$, 
    Corollary~\ref{cor:upper bounds for cubes} assigns $10$ vertices to each facet, each one doubled, so computes $32-20 = 12$ as the upper bound. The bound from  Theorem~\ref{thm:general upper bound} is $22$.

\subsection{Extremal designs and linear codes}
By Theorem~\ref{thm: gale duality and designs}, the smallest cardinality positively weighted extremal designs in $Q_d$ come from the facets of $P_m$ with the maximum number of vertices. Note that all elements of $\mathcal{U}_m$ are vertices of $P_m$. 
Here we use the theory of linear codes to improve the bounds in Corollary~\ref{cor:upper bounds for cubes}.
%We construct even stronger bounds on the size of a minimal extremal design through linear codes. 

\begin{definition}
Consider the Boolean field $\ZZ_2=\{0,1\}$ of integers mod $2$ and the 
 $\ZZ_2$-vector space $\{0,1\}^d$.
\begin{enumerate}
    \item A {\em linear code} $C \subseteq \{0,1\}^d$ is a subspace of 
    $\{0,1\}^d$. The {\em length} of the code $C$ is $d$ and its {\em dimension} is the dimension 
    of the subspace. 
    \item The {\em parity check matrix} of $C$ is a matrix $M \in \{0,1\}^{t \times d}$ such that 
$C = \ker(M)$. 
    \item The \emph{dual code} of $C = \ker(M)$ is the linear code $C^\perp := \textup{rowspan}(M)$. 
\end{enumerate}    
\end{definition}

We will rely heavily on the following result that connects linear codes to designs.

\begin{lemma}[Theorem 4.8 of \cite{cubescodes}] \label{lem: lin code check matrix}
Let $C = \ker(M)$ be a linear code in $\{0,1\}^d$. Then $C$ averages the eigenvector $\phi_x$, with  $\phi_x(y) = (-1)^{x^\top y}$, with equal weights if and only if $x \notin C^{\perp} \setminus \{0\} = \textup{rowspan}(M) \setminus \{0\}$.
\end{lemma}

If the matrix $M$ is chosen so that its row span is contained in the last eigenspace of $Q_d$ 
in any ordering, then by Lemma~\ref{lem: lin code check matrix}, $C = \ker(M)$ would be a combinatorial extremal design in $Q_d$ for that ordering. In fact, by Lemma~\ref{lem: comb complements}, both $C$ and its complement 
in $\{0,1\}^d$ would be combinatorial extremal designs in $\{0,1\}^d$. Equivalently, $C$ partitions the vertices of the extremal eigenpolytope into two faces.

\begin{lemma} \label{lem:there are combinatorial extremal designs}
The graph $Q_d$ has combinatorial extremal designs for any ordering.  
\end{lemma}

\begin{proof}
Let $\L$ be the last eigenspace of $Q_d$ in a given ordering. Then $\L = \{\phi_x: |x| = i\}$ for some $ i \in [d]$.
 Choose any $x\in \cal J_{d,i}$ and consider $W = \ker([x^\top])$. 
By Lemma \ref{lem: lin code check matrix}, $W$ averages (with equal weights) all eigenspaces of $Q_d$ except $\L$.
\end{proof}

\catherine{ This is essentially \cite[Theorem 3.3]{Golubev}, though that result emphasized that taking any vector as a check matrix provides an extremal design in {\em some} ordering, and here we emphasize that there is an extremal design for {\em any} ordering. Using Lemma~\ref{lem: lin code check matrix} it is possible to prove that the minimum cardinality extremal designs in Theorem~\ref{thm:cut polytope min designs} are combinatorial.}

We now return to frequency ordering and will prove in Theorem~\ref{thm:d=2 mod 4 theorem} that when $d \equiv 2$ mod $4$, there are positively weighted extremal designs of smallest size that are combinatorial. In the other cases, it is not clear whether the smallest positively weighted extremal designs in $Q_d$ are combinatorial. 

\begin{lemma} \label{lem: cube odd polytope max facet}
If $i$ is odd, then the eigenpolytope $P_i$ of $Q_d$ has at least $2 \binom{d}{i}$ facets 
each containing $2^{d-1}$ vertices. No facet of $P_{i}$ has more than $2^{d-1}$ vertices.
\end{lemma}

\begin{proof} 
By Lemma \ref{lem:odd and even eigenpolytopes}, $P_i$ is centrally symmetric. Therefore, the maximum number of vertices that can lie on a single facet of $P_i$ is $2^d/2 = 2^{d-1}$. We will use Gale duality and coding theory to exhibit $\binom{d}{i}$ facets which contain this many vertices.
Let $x \in \cal J_{d,i}$, and consider the parity check matrix $M = [x^\top]$.  By Lemma \ref{lem: lin code check matrix}, the code $C = \ker(M)$ averages all eigenspaces of $Q_d$ except for $\L_i$.  By Theorem \ref{thm: gale duality and designs}, $\{0,1\}^d \setminus C$ are the vertices on a face of $P_i$. Since the design is combinatorial, $C$ also indexes a face of $P_i$ by Lemma~\ref{lem: comb complements}. Since $|C| = 2^{d-1}$, $|\{0,1\}^d \setminus C| = 2^{d-1}$. Therefore, these faces contains the maximum possible number of vertices, and hence must be a facets. There are $\binom{d}{i}$ choices of the vector $x$. Each provides two unique facets --  $\{0,1\}^d \setminus C$ is not a linear code since it does not contain 0. \end{proof}

The code $ \ker(\ones^\top)$ is known as the {\em single parity check code}. Linear codes are said to be equivalent if they only differ by a permutation of coordinates, so the codes $\ker(x^\top)$ are equivalent for any $x \in \cal J_{d,i}$. In a sense, these codes are generalizations of the single parity check code, 
%though we cannot find a reference to them in coding theory, 
but they typically will have poor distance if $x\neq \ones$.

If the last eigenspace in frequency order is indexed by an even Hamming weight, then we can always do better.  
\begin{lemma} \label{lem: even wins}
If $i$ is even, then $Q_d$ has a combinatorial design that averages all but $\L_i$ with strictly fewer than $2^{d-1}$ vertices.
\end{lemma}

\begin{proof}
Consider a $2 \times d$ check matrix $M = [ x_1 \,\, x_2 ]^\top$ with $x_1, x_2 \in \cal J_{d,i}$ and $\supp x_1 \cap \supp x_2 = i/2$.  This is always possible; here is an example for $d = 9$, $ i = 4 $:
\[ \begin{bmatrix}
 1 & 1 & 1 & 1 & 0 & 0 & 0 & 0 & 0\\
 1 & 1 & 0 & 0 & 1 & 1 & 0 & 0 & 0
\end{bmatrix}
\]
Then, the row span of $M$ is contained in $\cal J_{d,i}$. By Lemma \ref{lem: lin code check matrix},  $\ker(M)$ averages all eigenspaces of $Q_d$ other than $\L_i$,  and $|\ker(M)| = 2^{d-2}$. 
\end{proof}

We now use the tools we have built so far to find or bound the size of the smallest positively weighted extremal designs in $Q_d$. Our constructions yield combinatorial designs based on linear codes. We begin by considering the case of $d \equiv 2$ mod $4$ for which Theorem~\ref{thm:d=2 mod 4 theorem} provides an optimal answer. 

\begin{theorem} \label{thm:d=2 mod 4 theorem}
Let $d \equiv 2 \mod 4$.  A minimum cardinality positively weighted extremal design of $Q_d$ in frequency order consists of $2^{d-1}$ vertices and is combinatorial.  Any code $C = \ker(x^\top)$ for $x\in \cal J_{d,d/2}$ attains this minimum.
\end{theorem}

\begin{proof}
The last eigenspace of $Q_d$ by frequency is $\L_{d/2}$, and $d/2$ is odd.  By Lemma \ref{lem: cube odd polytope max facet}, the extremal eigenpolytope $P_{d/2}$ has facets with $2^{d-1}$ vertices, the maximum possible, and 
these vertices are all elements of $\cal U_{d/2}$. Therefore, a minimum positively weighted 
extremal design in $Q_d$ consists of $2^{d-1}$ elements by Theorem~\ref{thm: gale duality and designs}.  Every linear code of the form $C = \ker(x^\top)$ where $x \in \mathcal{J}_{d,d/2}$ is a combinatorial design that achieves this minimum by Lemma~\ref{lem: lin code check matrix}.
\end{proof}

Next we consider $d \not\equiv 2 \mod 4.$ In these cases, it follows from Lemmas~\ref{lem: cube odd polytope max facet} and \ref{lem: even wins} that when there is a tie for the last eigenspace, 
the smallest extremal designs are going to come from the extremal eigenpolytope with an even index. Let $m$ be this index in the rest of this section. Concretely, 
$$ 
m = \left\{ \begin{array}{ll} 
d/2 & \textup{ if } d \equiv 0 \textup{ mod } 4 \\
(d-1)/2 & \textup{ if } d \equiv 1 \textup{ mod } 4 \\
(d+1)/2 & \textup{ if } d \equiv 3 \textup{ mod } 4 .\\
\end{array} \right.
$$

We formalize the comments stated after Lemma~\ref{lem: lin code check matrix} for the index $m$. 

\begin{corollary}\label{cor:constant weight}
If the non-zero vectors in the row span of $M$ are entirely of weight $m$, then the linear code
$C = \ker (M)$ is an extremal combinatorial design in $Q_d$.
\end{corollary}
\begin{proof}
By Lemma \ref{lem: lin code check matrix}, the code $C$ averages all eigenvectors of $Q_d$ except for those indexed by non-zero elements in the row span of $M$. Since these elements are contained in $\cal J_{d,m}$, the corresponding eigenvectors all lie in $\Lambda_m$. Hence $C$ is an extremal combinatorial design of $Q_d$.
\end{proof}

Row spans as in Corollary~\ref{cor:constant weight}, 
in which all non-zero elements have the same Hamming weight,
are called {\em linear equidistant codes} or \emph{constant weight linear codes} \cite{BlakeMullinCodingTheory}. These codes have been completely classified by Bonisoli's theorem \cite{Bonisoli}; we refer the reader also to \cite{WardDivisibleCodes}. We first recall some facts about linear codes.  The binary \emph{Hamming code} $H_r \subset \{0,1\}^{2^r -1}$ \cite{HammingOG} is the linear code whose check matrix $M_r \in \{0,1\}^{r \times (2^r -1)}$ has columns consisting of the binary expansions of the digits $\{1,\ldots, 2^r-1\}$. For instance, the check matrix of $H_3 \subset \{0,1\}^7$ is  \[M_3= \begin{bmatrix} 
1& 0 & 1 & 0& 1 &0& 1 \\
 0 & 1 & 1 & 0 & 0 & 1 & 1 \\
 0 & 0 & 0 & 1 & 1 & 1 & 1 \\
\end{bmatrix}. \]

The dual of the Hamming code $H_r^\perp$, i.e., the row span of $M_r$, is called the {\em simplex code}, so named because its vectors form the vertex set of a regular $(2^r-1)$-simplex. Every non-zero element in the simplex code $H_r^\perp$ has weight $2^{r-1}$. 

\begin{theorem}[\cite{Bonisoli}]
If $C$ is a $r$-dimensional linear equidistant code, then $C$ is equivalent to concatenated copies of the simplex code $H_r^\perp$, possibly with additional zero coordinates. 
\end{theorem}

This means that an $r$-dimensional linear equidistant code is equivalent to the row span of 
a maximal concatenation of the check matrix $M_r$ of $H_r$ with possibly additional zero columns. 
For instance, a $3$-dimensional linear equidistant code in $\{0,1\}^{15}$ is equivalent to the row span of $[\, \, M_3\, \, |\, \, M_3 \, \,|\, \, \vec 0 \, \,]$.

\begin{lemma}\label{lem: max halfweight code}
For $d \equiv 3 \mod 4$, let $d+1 = 2^t \cdot b$ with $t$ as big as possible. Then $t$ is the maximum dimension of a $((d+1)/2)$-weight linear equidistant code in $\{0,1\}^d$, $\{0,1\}^{d+1}$, and $\{0,1\}^{d+2}$. 
\end{lemma}

\begin{proof}
We first show that there are $t$-dimensional linear equidistant codes of lengths $d, d+1,$ and $d+2$. Consider the matrix $M$ given by concatenating $b$ copies of the check matrix $M_{t}$. Then $M \in \{0,1\}^{t \times (b(2^{t} -1))}.$  Note that $$b(2^{t} -1) = b2^t - b = d + 1 - b.$$
Since every non-zero element in the row span of $M_t$ has weight $2^{t-1}$, every non-zero element in the row span of $M$ has weight 
$$ b2^{t-1} = (d+1)/2.$$
Thus by appending $b-1$ columns of zeros to $M$, we arrive at a matrix in $\{0,1\}^{t \times d}$ in which every non-zero element lies in $\cal J_{d,\frac{d+1}{2}}$. Appending one or two more 
zero columns creates matrices for which all non-zero elements in their row spans are contained in $\cal J_{d+1,\frac{d+1}{2}}$ or $\cal J_{d+2,\frac{d+1}{2}}$, respectively. Note that if 
$d \equiv 3$ mod $4$, then $\frac{d+1}{2}$ is even and indexes the 
extremal eigenspace of interest in $Q_d, Q_{d+1}$ and $Q_{d+2}$. 
%Hence, by Corollary~\ref{cor:constant weight}, $\ker(M)$ is an extremal combinatorial design in these graphs.

We claim that $t$ is the maximum possible dimension of a $((d+1)/2)$-weight linear equidistant code in $\{0,1\}^d$, $\{0,1\}^{d+1}$, and $\{0,1\}^{d+2}$. Let $T > t$, and suppose there is a $T$-dimensional linear equidistant code of length $d$ and weight $(d+1)/2$.  Then this code is the row span of $q$ copies of the Hamming check matrix $M_T$, possibly padded with some columns of zeros.  The weight of each row is $q 2^{T-1}$, so $(d+1)/2 = q 2^{T-1}$ implies that $d= 2^T q - 1$.  Since we also know $ d = 2^t b -1 $, it follows that 
\begin{align*}
    2^Tq -1  = d  = 2^t \cdot b -1  \iff b = 2^{T-t} q
\end{align*}
Since $T > t$, this implies that $b$ is even, which contradicts the definition of $t$. 
A similar argument also works for $d+1$ and $d+2$. 
\end{proof}

\begin{theorem} \label{thm: Qd min codes}
For each triple $(d,m,t)$ shown below, the smallest positively weighted 
extremal designs of $Q_d$ have  at most  $2^{d-t}$ elements and are obtained by choosing $\L_m$ to be last in frequency order.
\begin{enumerate}
    \item $d \equiv 0$ mod $4$: $m = d/2$ and $d=2^t \cdot b$ with $t$ maximal.
    \item $d \equiv 1$ mod $4$, $m = (d-1)/2$ and $d-1=2^t \cdot b$ with $t$ maximal.
    \item $d \equiv 3$ mod $4$, $m =  (d+1)/2$ and $d+1=2^t \cdot b$ with $t$ maximal.
\end{enumerate}
\end{theorem}

\begin{proof}
Note that for each $d$ shown above, the corresponding $m$ in the triple is even and 
indexes the extremal eigenspace of $Q_d$ (if there is a tie) 
that can yield the smallest positively weighted 
designs. This follows from Lemmas~\ref{lem: cube odd polytope max facet} and \ref{lem: even wins}.
By Lemma~\ref{lem: max halfweight code}, there is a maximum cardinality 
linear equidistant code $C \subset \{0,1\}^d$ of dimension $t$ and weight $m$. It then follows by the strategy in Corollary~\ref{cor:constant weight} that the dual code $C^\perp$ is then an extremal combinatorial design in $Q_d$, and $|C^\perp| = 2^{d-t}$.  
\end{proof}

We attribute Theorem~\ref{thm: Qd min codes} to Chris Lee and David Shiroma 
who discovered these bounds in an undergraduate project supervised by the authors. Following the strategy outlined in Corollary~\ref{cor:constant weight}, they discovered the construction in Bonisoli's theorem from which the result follows.  Table~\ref{tab: cube bounds small d} computes the bounds in Theorems~\ref{thm:d=2 mod 4 theorem} and \ref{thm: Qd min codes} for small values of $d$. 

\begin{table}[h]
\begin{center}
\begin{tabular}{c|c|c|c|c | c}
   $d$  & $m$ & $\dim(P_m) = {d \choose m}$ & $\#V(P_m)$  &  $|W^\ast| \leq $ &  $\# V(F^\ast) \geq$\\
    \hline
    2  & 1 & 2 & 4 & 2 & 2 \\
    \hline
    3  & 2 & 3 & $4 \cdot 2 $& 2 & 6\\
    \hline
    4  & 2 & 6 & $8\cdot 2$  & 4 & 12\\
    \hline
    5 & 2 & 10 & $16 \cdot 2$ &  8 & 24\\
    \hline
    6 & 3 & 20 & 64 & 32 & 32\\
    \hline
   7& 4 & 35 & $64 \cdot 2$ & 16 & 112\\
    \hline
    8 & 4 & 70 & $128 \cdot 2$& 32 & 224 \\
    \hline
    9 & 4 & 126 & $256 \cdot 2$ &  64 & 448\\
    \hline
     10  & 5 & 252 & 1024 &  512 & 512\\
    \hline
    11   & 6 & 462 & $1024 \cdot 2$ & 512 & 1536
\end{tabular}
\end{center}
    \caption{Bounds from Theorems  \ref{thm:d=2 mod 4 theorem},\ref{thm: Qd min codes} for small values of $d$. $W^\ast$ denotes a minimum cardinality positively weighted extremal design, and $F^\ast$ denotes a facet of $P_m$ with the maximum number of vertices including multiplicity. When $P_m$ has $N$ distinct vertices which are doubled up, we write $\#V(P_m) = N \cdot 2$. }
    \label{tab: cube bounds small d}
    \vspace{-.1 in}
\end{table}

Our main strategy in this paper has been to use the facet combinatorics of extremal eigenpolytopes to find the smallest positively weighted extremal designs in graphs. In the case of hypercubes, 
this strategy worked for 
$Q_d$ when $d \equiv 2$ mod $4$. In the other cases, it was much harder to understand the facets of the extremal eigenpolytope $P_m$, and instead we found small designs using the theory of linear codes. By Theorem~\ref{thm: gale duality and designs}, 
these small designs correspond to some faces of the extremal eigenpolytope $P_m$.

\begin{corollary}
In each of the following situations, the extremal eigenpolytope $P_m$ of $Q_d$ 
has a face containing $2^d - 2^{d-t}$ vertices:
\begin{enumerate}
    \item $d \equiv 0$ mod $4$, $m =  d/2$  $d=2^t \cdot b$ with $t$ maximal.
    \item $d \equiv 1$ mod $4$, $m =  (d-1)/2$ and $d-1=2^t \cdot b$ with $t$ maximal.
    \item $d \equiv 3$ mod $4$, $m =  (d+1)/2$ and $d+1=2^t \cdot b$ with $t$ maximal.
\end{enumerate}
\end{corollary}

\begin{proof}
Combinatorial designs are positively weighted designs. Thus the extremal designs of Theorem~\ref{thm: Qd min codes} provide these faces by Gale duality.
\end{proof}

We conjecture that the 
bounds in Theorem~\ref{thm: Qd min codes} are optimal.

\begin{conjecture}
The duals of the constant weight codes constructed in Lemma~\ref{lem: max halfweight code} are smallest cardinality positively weighted extremal designs in their $Q_d$.
\end{conjecture}

To prove this conjecture, it would suffice to prove the the faces of $P_m$ given by 
these dual codes are (i) facets of $P_m$, and (ii) contain the most vertices among all facets of $P_m$.
We note these faces of $P_m$ contain an enormous number of vertices.  

% To prove 
% that such a face is a facet, it would suffice to show the following.

% \begin{conjecture}
% For $d$ and $m$ as in Theorem~\ref{thm: Qd min codes}, the 
% submatrix of $U_m$ consisting of columns not indexed by the dual 
% of the maximum constant weight code in Lemma~\ref{lem: max halfweight code} has 
% full row rank ${d \choose m}$.
% \end{conjecture}

We have relied on linear codes to find small designs in $Q_d$ 
when $d \not\equiv 2$ mod $4$. However, there is no reason to believe that the 
smallest, or all minimal, positively weighted extremal designs in such $Q_d$ 
are codes. Indeed, when $d \not\equiv 2$ mod $4$, there are minimal positively weighted extremal designs of $Q_d$ that are not isomorphic to linear codes or their complements.

% \begin{lemma}
% When $d \not\equiv 2$ mod $4$, there are minimal positively weighted extremal designs of $Q_d$ that are not isomorphic to linear codes or their complements.
% \end{lemma}

\begin{example}
%When $d=5 we have m=2$. 
The extremal eigenpolytope $P_2$ of $Q_5$ has dimension $10={5 \choose 2}$. 
It has $16$ vertices each labeled by two columns of $U_2$ 
since $c(y) = c(\ones-y)$ for all 
$y \in \{0,1\}^5$. This polytope has $56$ facets that come in two symmetry classes; $16$ of them 
are simplices ($\Delta_9$) each containing $20=2 \cdot 10$ columns of $U_2$ as vertices, 
while the remaining $40$ facets each have $24=2 \cdot 12$ vertices.
The minimal design complementary to a simplex facet contains $32- 20 = 12$ vertices of $Q_5$.  
A linear code of length $5$ must have size $2^t$ for $t \in [5]$. Since neither $12$ nor $32-12$ are powers of $2$, such a design is not isomorphic to a linear code or its complement. 
\end{example}

\printbibliography

\end{document}